\documentclass[oneside,a4paper,10pt]{amsart}
\usepackage{geometry} 
\usepackage{graphicx}
\usepackage{amsmath} 
\usepackage{amsthm, amssymb, amsfonts} 
\usepackage{color} 
\usepackage{overpic} 
\usepackage{contour} 
\contourlength{1pt}
\usepackage[bookmarksnumbered=true]{hyperref} 
\hypersetup{
     colorlinks = true,
     linkcolor = blue,
     anchorcolor = blue,
     citecolor = teal,
     filecolor = blue,
     urlcolor = blue
     }
\usepackage{tikz} 
\usepackage{enumerate}
\usepackage{enumitem}
\setlist[enumerate, 1]{label=(\roman*)}

\def\centerarc[#1](#2)(#3:#4:#5);%
{

    \draw[#1]([shift=(#3:#5)]#2) arc (#3:#4:#5);
}

\DeclareMathOperator{\SL}{SL} 
\DeclareMathOperator{\GL}{GL}

\newcommand{\Id}{{\rm Id}}
\newcommand{\Il}{{\rm I}\ell}
\newcommand{\IA}{{\rm IA}}
\newcommand{\IV}{{\rm IV}}
\newcommand{\isin}{{\rm Isin}}

\newcommand{\Span}{{\rm span}}

\newcommand{\ZZ}{\mathbb{Z}}

\definecolor{myblue}{rgb}{0.11,.19,.8}
\definecolor{myred}{rgb}{0.8,.014,.04}
\definecolor{mygreen}{rgb}{0.1,.64,.1}

\usepackage[todonotes={textsize=footnotesize}]{changes}
\definechangesauthor[color=red]{CM}
\definechangesauthor[color=blue]{BS}
\definechangesauthor[color=mygreen]{OK}
\definechangesauthor[color=cyan]{FM}

\newtheorem{theorem}{Theorem}[section]
\newtheorem{proposition}[theorem]{Proposition}

\theoremstyle{definition}
\newtheorem{definition}[theorem]{Definition}
\newtheorem{example}[theorem]{Example}
\newtheorem{remark}[theorem]{Remark}

\makeatletter\def\@captionfont{\normalfont\footnotesize}\makeatother

\title{Klein-Arnold tensegrities}

\author{Oleg Karpenkov, 
Fatemeh Mohammadi, 
Christian M{\"u}ller and 
Bernd Schulze}

\begin{document}

\maketitle

\begin{abstract} 
In this paper, we introduce new classes of infinite and combinatorially periodic tensegrities, derived from algebraic multidimensional continued fractions in the sense of F.~Klein. We describe the stress coefficients on edges through integer invariants of these continued fractions, as initiated by V.I.~Arnold, thereby creating a novel connection between geometric rigidity theory and the geometry of continued fractions. Remarkably, the new classes of tensegrities possess rational self-stress coefficients. To establish the self-stressability of the frameworks, we present a projective version of the classical Maxwell-Cremona lifting principle, a result of independent interest.
\end{abstract}

{\hypersetup{linkcolor=black}

{\tableofcontents}}

\section*{Introduction}

The goal of this paper is to introduce and study new types of tensegrities \cite{conten,connellyguest,connix,conwhi} with rational stress coefficients that arise from a modified Maxwell-Cremona projection of periodic Klein continued fractions (Theorems~\ref{thm:self-stress} and~\ref{thm:stresscoefprop}). It turns out that these self-stresses have a natural expression in terms of invariants of integer geometry (see Proposition~\ref{integer-propos}). We thereby create  a link between the areas of (geometric) rigidity theory, graphic statics, lattice geometry and the geometry of continued fractions.

Rigidity theory 
is concerned with  analysing the rigidity and flexibility of discrete configurations of  primitive geometric objects (points, lines, etc.) that satisfy a set of geometric constraints  (fixed lengths, directions, etc.). 
See \cite{connellyguest,graver} for introductions to this topic and \cite{berndwalter} for a
 summary of recent results.  
 The paradigmatic example is that 
of a \emph{bar-joint framework} in Euclidean $n$-space, $\mathbb{R}^n$, which is defined as 
$G(p)$, 
where $G=(V,E)$ is a graph (with $E$ and $V$ representing fixed length bars and freely rotating joints, respectively)
 and $p:V\to \mathbb{R}^n$ is a map that describes the geometric placement of the joints. 
The theory has a rich history, which can be traced back to  work of L.~Euler, A.-L.~Cauchy and J.C.~Maxwell.  Driven by numerous modern  applications in fields as diverse as engineering, robotics, structural biology, crystallography, 
material science, and CAD, and by accompanying significant mathematical breakthroughs, interest and developments in the
rigidity of bar-joint frameworks and related geometric constraint systems have
increased rapidly in recent decades.

Methods for recognising whether a bar-joint framework is rigid often rely on the concept of \emph{first-order} or 
\emph{infinitesimal rigidity} \cite{connellyguest,graver,WW}. This is a linearised version of rigidity that is equivalent to \emph{static rigidity}, a fundamental
notion in structural engineering. Loosely speaking, a bar-joint framework is \emph{statically rigid} if every equilibrium load (that is, any assignment
 of forces to the joints that do not induce a rigid-body motion of the framework) can be resolved by forces
 (tensions or compressions) within the bars. If there is a set of non-zero forces within the bars that are in equilibrium in the absence of
any external load, then the framework is said to have a non-trivial \emph{equilbrium stress} or \emph{self-stress}, and is
often referred to as a \emph{tensegrity}. Self-stresses indicate
redundancies in the bar constraints and play a fundamental role in the rigidity analyses of bar-joint frameworks/tensegrities.

Based on work by W.J.M.~Rankine \cite{ran} and P.~Culmann \cite{cul}, among others, it was discovered by J.C.~Maxwell and L.~Cremona in the 1860s that for a plane framework, there is an equivalence between self-stresses,
reciprocal frameworks on the dual graph, and vertical  liftings of the original framework to polyhedral surfaces \cite{maxwell1864xlv,maxwell1870reciprocal,cremona1872figure}. See also \cite{WW82,CW93}. This provides a tool to describe
and analyse self-stresses (and infinitesimal motions) in terms of polyhedral
liftings, which has found important applications in discrete and computational
geometry (see
e.g.~\cite{Con03,crapo1994spaces,hopcroft1992paradigm,richter2006realization,aschulz1}).
This ``Maxwell-Cremona correspondence" and related
geometric methods from the area of ``graphic statics" have also received much attention from the engineering and architecture communities lately,
as they provide visual and intuitive approaches for the design, analysis, and optimisation of structures via modern computational tools (see e.g.~\cite{zal,begh,block,mbmm}).

Since symmetry and periodicity are ubiquitous in both man-made and natural structures, the last two decades have seen an explosion of results regarding the  infinitesimal/static rigidity  of symmetric and periodic frameworks \cite{SWbook}, and the design of symmetric structures with certain rigidity properties \cite{connellysw,kst21}. Here we will introduce tensegrities  with a new type of hyperbolic periodic symmetry in the plane. 

To this end we employ a multidimensional version of continued fractions
introduced by F.~Klein in 1895~\cite{Klein1895, Klein1896}. Despite its name,
Klein continued fractions are actually not numbers but rather  convex
polyhedral surfaces naturally defined by hyperbolic matrices (with real
distinct eigenvalues). This already suggests their possible applications in
rigidity theory via the  Maxwell-Cremona correspondence. Dirichlet's theorem on
units (see e.g.~\cite{BS66}) implies that if a hyperbolic matrix is algebraic
(i.e., if it is in $\GL(n,\ZZ)$) then the corresponding Klein continued
fractions are combinatorially $(n{-}1)$-periodic.

Due to the computational complexity, the first periods of Klein periodic continued fractions only appeared a hundred years later when the computational power had become sufficiently strong. Around 1990, V.I.~Arnold formulated several milestone problems and conjectures on the computation of periods of such continued fractions, their combinatorics and their dynamical properties (see \cite{Arnold1998}). The first examples
were computed by E.~Korkina~\cite{Kor95}, A.D.~Bruno and V.I.~Parusnikov~\cite{BP94,Par95}, O.~Karpenkov~\cite{Kar04}, among others.
For further details, including algorithms and discussions, we refer the reader to~\cite{karpenkov-book}.

To obtain our new class of ``Klein-Arnold tensegrities" from Klein continued fractions, we  describe a projective version of the classical  Maxwell-Cremona correspondence which uses central, rather than orthogonal, projections. 
A fundamental fact, which was already known to scientists and mathematicians such as Rankine, Maxwell and Klein, is that
static rigidity  is projectively invariant \cite{CW1982,Izmestiev1,NSW}. In particular, if two frameworks are projectively equivalent,
 then the spaces of their self-stresses have the same dimension. The values of their stress coefficients, however, can be different.
Therefore,  it is natural to consider force loads as objects of projective geometry (rather than affine geometry). This is described in Section~\ref{sec:projstatics}. In Section~\ref{sec:proofss} (proof of Theorem~\ref{thm:self-stress}) we then establish 
an explicit formula for the  stress coefficients of the  tensegrities obtained from centrally projecting a polyhedral surface to an affine plane. This is of interest with regards to applications of graphic statics in its own right. As we will show in Section~\ref{sec:projective-stress},  the stress coefficients  can be expressed in terms of invariants of integer geometry. For projections of Klein continued fractions, we obtain tensegrities with the remarkable property that  they have doubly periodic rational self-stresses (Theorem~\ref{thm:stresscoefprop}).

Finally, we would like to mention  a recent paper by F.~Mohammadi and X.~Wu~\cite{mohammadi2024rational}
providing a link between rigidity theory and Chow cohomologies of toric varieties. Our paper provides a combinatorial basis for the extension of this link.

We collect classical notions of geometric rigidity theory in Section~\ref{Preliminaries}.
Further, in Section~\ref{sec:self-stresses} we define self-stresses for projections of polyhedral surfaces to affine planes.
In Section~\ref{Projective-proof} we provide the proof of Theorem~\ref{thm:self-stress}.
In the proof, we use the notion of projective self-stresses; these self-stresses arising from Klein continued fractions are periodic with respect to a Dirichlet group action. 
In Section~\ref{sec:projective-stress},
we discuss the necessary definitions of integer geometry and describe the stress coefficients  in terms of invariants of integer geometry. We discuss multidimensional continued fractions in the sense of Klein in Section~\ref{sec:self-stresses-klein-sails}. We explain their combinatorial periodicity in the algebraic case and show how to construct corresponding periodic self-stresses in the plane. We illustrate this construction with a couple of examples.

\section{Preliminaries from geometric rigidity theory and graphic statics} 
\label{Preliminaries}

We collect some basic notions from geometric rigidity theory and graphic
statics. For further details, see
e.g.~\cite{connellyguest,connelly1996second,white1983algebraic,WW}.

Throughout the paper, we let $G = (V, E)$ be a graph without loops and without
multiple edges where $V = \{v_1, v_2, \ldots\}$ denotes the set of
\emph{vertices} and $E$ the set of \emph{edges} of $G$. An edge joining the
vertices $v_i$ and $v_j$ is denoted by $v_i v_j$.
A \emph{(bar-joint) framework} $G(p)$ in $\mathbb{R}^n$ consists of a graph $G =
(V, E)$ and an injective map $p : V \to \mathbb{R}^n$. We often denote $p(v_i)$ by
$p_i$. We say that there is an \emph{edge} between $p_i$ and $p_j$ if $v_i v_j$
is an edge of $G$ and denote it by $p_i p_j$. 
      
A \emph{stress} on a framework $G(p)$ is a map $\omega:E\rightarrow\mathbb{R}$, where $\omega(v_i v_j)$ is usually denoted by $\omega_{ij}$. Note that $\omega_{ij}=\omega_{ji}$. The scalars $\omega_{ij}$ are called \emph{stress coefficients} (or \emph{force-densities} by structural engineers) and represent \emph{tensions} and \emph{compressions} in the corresponding edges depending on their sign.
A stress $\omega$ is called an \emph{equilibrium stress} or \emph{self-stress} if for every $v_i \in V$, we have
\begin{equation*}
  \sum\limits_{\{v_j \mid v_j \neq v_i\}} \omega_{ij} (p_i - p_j) = 0.
\end{equation*}

A pair $(G(p), w)$ is called a \emph{tensegrity} if $\omega$ is a self-stress of the framework $G(p)$.
A tensegrity $(G(p), w)$ (or self-stress $\omega$) is said to be \emph{non-zero} if there exists an edge $v_i v_j$ of the framework that has a non-zero stress coefficient $\omega_{ij} \neq 0$.

We say that a graph $G$ is \emph{polyhedral} if it is vertex 3-connected and planar (and hence is the graph of a $3$-dimensional convex polytope, by Steinitz's Theorem \cite{steinitz,gru}). Note that, as such, 
$G$ has a uniquely determined set of faces, namely the set of induced non-separating cycles of $G$. 

For a polyhedral graph $G$, a \emph{(vertical) lifting} or \emph{Maxwell-Cremona lifting} of a framework $G(p)$ in the plane is a function $h:V\to \mathbb{R}$ that assigns to each vertex $p_i$ of $G(p)$ the additional $z$-coordinate (or ``height'') $z_i=h(p_i)$, so that in the resulting $3$-dimensional framework, every face of $G$ lies within a plane.
The following remarkable correspondence was found by L.~Cremona and
J.C.~Maxwell~\cite{maxwell1870reciprocal,cremona1872figure} in the 19th
century of which a complete proof was given in~\cite{WW82} (see also~\cite{CW93}). 

\begin{theorem}[\textbf{Maxwell-Cremona correspondence}]
  \label{thm:maxwell}
  Let $G$ be a polyhedral graph, $G(p)$ be a framework in the $xy$-plane, and
  $f$ be a designated face of $G$. Then there is a one-to-one correspondence
  between self-stresses $\omega$ of $G(p)$ and vertical liftings of $G(p)$ in
  $\mathbb{R}^3$, where the face $f$ lies in the $xy$-plane.
\end{theorem}

Note that while $G$ is planar, the framework $G(p)$ may have crossing edges. It turns out that the  sign of the stress coefficient of an edge is related to the connectivity of the two faces along the corresponding edge of the  polyhedral surface in $3$-space:  positive coefficients correspond to convex  and negative coefficients to concave dihedral angles. 

Let $P$ be a given $3$-polytope and let $f$ be a designated face of $P$ with the
property that the orthogonal projection of $P$ onto the plane of $f$ maps every
point of $P$ that is not in $f$ to the interior of $f$. Then Theorem~3.3 in Lov\'asz's book~\cite{lov} shows how to obtain the stress-coefficients of the
corresponding self-stressed framework. Such a framework is also known as a
``rubber-band representation'' of the polytope, based on Tutte's spring theorem
\cite{tutte}. (Conversely, the lifting may be constructed from the stress
coefficients, as was known by Maxwell and Cremona. See Section~13 of
\cite{richter2006realization} for a simple algorithm.) In
Section~\ref{sec:self-stresses}, we will provide a corresponding formula for the
stress coefficients arising from central, rather than orthogonal, projections
of polyhedra.

\section{Self-stresses defined by projections of polyhedral surfaces}
\label{sec:self-stresses}

In this section, we define self-stresses for central projections of polyhedral surfaces to affine planes.
In what follows we will only work in $\mathbb{R}^3$.
First of all we define projective lifting coefficients on edges which will
play a central role in the construction of self-stresses.
Let us denote by $\Span(W)$ the minimal affine subspace containing any collection $W$ of subsets of $\mathbb{R}^n$.

\begin{definition}
  \label{def:stresscoef}
  Let $p_1, \ldots, p_4$ be four points in $\mathbb{R}^3$ such that the affine spaces
  $\pi_1 = \Span(p_1, p_2, p_3)$ and $\pi_2 = \Span(p_1, p_2, p_4)$ are both
  $2$-dimensional and do not pass through the origin $O$. 
  Then we call the expression 
  \[
    \omega(p_1,p_2;p_3,p_4)
    =
    \frac{\det(p_2-p_1,p_3-p_1,p_4-p_1)}{\det(p_1,p_2,p_3) \det(p_1,p_2,p_4)}
  \] 
  a \emph{projective lifting coefficient} on $p_1 p_2$ induced by
  $\pi_1$ and $\pi_2$. 
  Note that $\omega(p_1,p_2;p_3,p_4) = -\omega(p_1,p_2;p_4,p_3)$.
\end{definition}

The following statement follows directly from the definition by applying elementary matrix manipulations.

\begin{proposition}
  \label{prop-does-not-depend}
  The projective lifting coefficient $\omega(p_1,p_2;p_3,p_4)$ does not depend
  on the choice of points $p_3\in \pi_1\setminus \pi_2$ and
  $p_4 \in \pi_2 \setminus \pi_1$.
\end{proposition}

Our next goal is to adapt Definition~\ref{def:stresscoef} to the edges of
oriented polyhedral surfaces. We start with the definition of projectively
non-degenerate surfaces and admissible projections thereof.

\begin{definition}
  Let $S(p)$ be an oriented polyhedral 2-surface in $\mathbb{R}^3$. 
  We say that $S(p)$ is \emph{projectively non-degenerate} if the faces of
  $S(p)$ span affine 2-spaces that do not include the origin $O$. 
  We say that an affine plane $\pi$ is a \emph{proper projection plane} for
  $S(p)$ if 
  \begin{itemize}
    \item $\pi$ does not contain the origin $O$;
    \item the lines $\Span(O, p_i)$ are not parallel to $\pi$ for all vertices $p_i$. 
  \end{itemize}
\end{definition}

There is a natural framework induced by projecting the 1-skeleton of a  polyhedral surface to a proper projection plane.

\begin{definition}
  \label{def-betas}
  For a projectively non-degenerate polyhedral surface $S(p)$ and a proper
  projection plane $\pi$, we consider the framework with
  \begin{itemize}
    \item vertices $\bar p_i = \Span(O, p_i) \cap \pi$;
    \item edges $\bar p_i \bar p_j$
      corresponding to the edges $p_ip_j$ in $S(p)$.
  \end{itemize}
  We call it the \emph{projection-framework} of $S(p)$ to the proper projection
  plane $\pi$ and denote it by $G_{S}(\bar p)$ (see Figure~\ref{fig:projection}
  left).
  We call the scaling factor $\beta_i$ of the parallel vectors 
  \begin{equation*}
    p_i - O = \beta_i \cdot (\bar p_i - O),
  \end{equation*}
  the \emph{projection-coefficient} for the vertex $p_i$.
\end{definition}

\begin{center}
\begin{figure}[htp]
  \begin{overpic}[width=.5\textwidth]{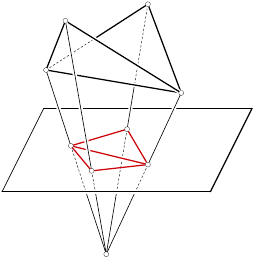}
    \put(14,75){$p_1$}
    \put(72,64){$p_2$}
    \put(23,94){$p_3$}
    \put(59,98){$p_4$}
    \put(22,43){\color{myred}$\bar p_1$}
    \put(59,35){\color{myred}$\bar p_2$}
    \put(30,33){\contour{white}{\color{myred}$\bar p_3$}}
    \put(51,49){\color{myred}$\bar p_4$}
    \put(35,0){$O$}
    \put(92,54){$\pi$}
    \put(65,82){$S(p)$}
    \put(35,50){\small\color{myred}{$G_{S}(\bar p)$}}
  \end{overpic}
  \hfill
  \begin{overpic}[width=.4\textwidth]{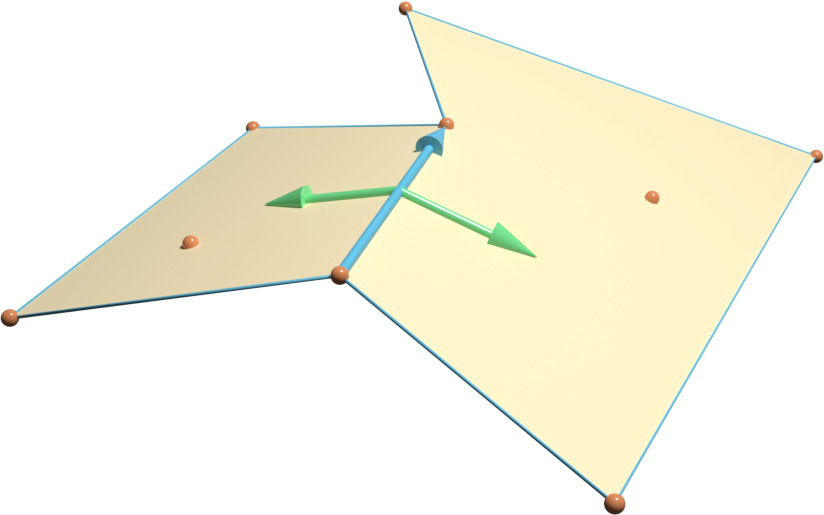} 
    \put(37,24){\small\contour{white}{$p_i$}}
    \put(56,46){\small\contour{white}{$p_j$}}
    \put(77,41){\small\contour{white}{$p_k$}}
    \put(17,30){\small\contour{white}{$p_l$}}
    \put(36,33){\scriptsize\rotatebox{6}{\contour{white}{left}}}
    \put(49,34){\scriptsize\rotatebox{-27}{\contour{white}{right}}}
  \end{overpic}
  \caption{\emph{Left:} Central projection of a polyhedral surface $S(p)$ with
  vertices $p_i$ to a proper projection plane $\pi$. The projected framework
  $G_S(\bar p)$ with vertices $\bar p_i$ is a tensegrity.
  \emph{Right:} Two faces adjacent to an oriented edge $\overrightarrow{p_i
  p_j}$. The given orientation of the surface specifies a left and right face
  with respect to an oriented edge.
  The projective lifting coefficient $\omega_{ij}$ is determined per edge
  independent from the orientation of the edge.}
  \label{fig:projection}
\end{figure}
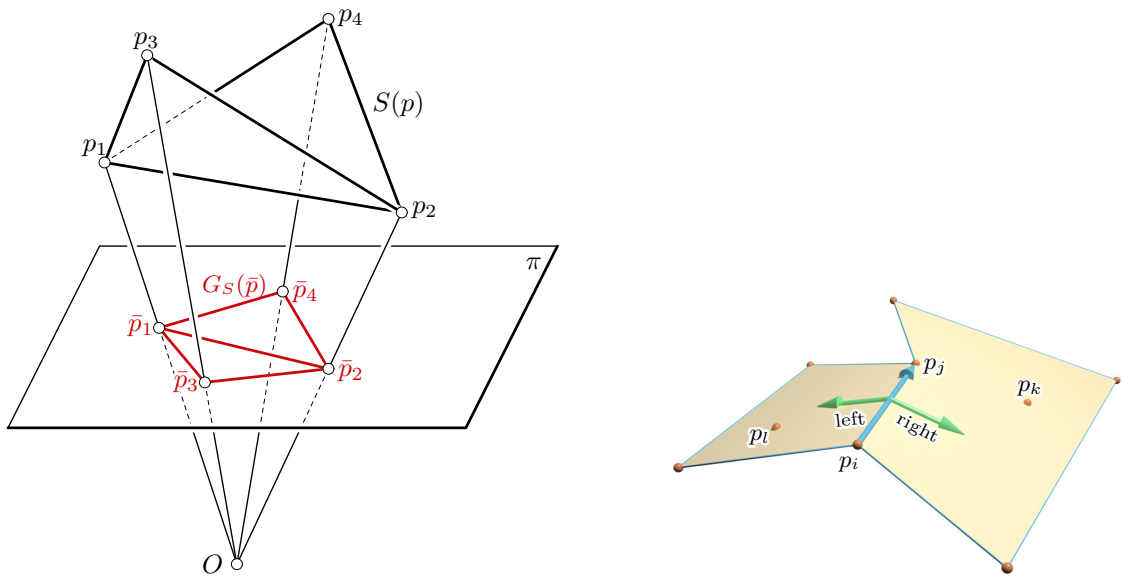
\end{center}

In the next two definitions we introduce a natural projective version of Maxwell-Cremona self-stresses.

\begin{definition}
  \label{oriented-def-stress}
    Consider a projectively non-degenerate oriented polyhedral surface $S(p)$.
    For an oriented edge $\overrightarrow{p_i p_j}$ the orientation of the
    polyhedral surface determines a left and a right face adjacent to this
    edge. Let us choose a point $p_k$ on the right face and a point $p_l$ on
    the left face (see Figure~\ref{fig:projection} right).
    Then we say that the projective lifting coefficient
    $\omega_{ij} = \omega(p_i,p_j;p_k,p_l)$ is \emph{associated} to $S$ and
    we denote the collection of all $\omega_{ij}$ by $\omega_S$.
\end{definition}

\begin{definition}
  \label{def-omega-bar}
  Consider a projectively non-degenerate orientable polyhedral surface $S(p)$,
  a proper projection plane $\pi$, and the projection-framework $G_{S}(\bar p)$
  in $\pi$ together with the projection-coefficients $\beta$. We define the following notions:
  \begin{itemize}
    \item For each edge $\bar p_i\bar p_j$ of $G_{S}(\bar p)$ we set
      the coefficient 
      \begin{equation}
        \label{eq:projection-stress}
        \bar\omega_{ij}=\beta_i\beta_j\omega_{ij}
      \end{equation}
      and call it the \emph{projection-stress coefficient} of the edge.
    \item The collection $\bar\omega_S$ of the projection-stress
      coefficients of all edges is called the \emph{projection-stress} on the
      framework $G_S(\bar p)$.
    \item We also say that the surface $S(p)$ itself is a \emph{projective
      lifting} of the framework $G_S(\bar p)$.
  \end{itemize}
\end{definition}

\begin{remark}
  Note that in the case of the proper projection plane $\pi$ defined by the
  equation $z=1$, the projection-coefficient $\beta_i$ corresponding to vertex
  $p_i = (x_i, y_i, z_i)$ reads
  \begin{equation*}
    \beta_i=z_i.
  \end{equation*}
\end{remark}

Similar to the case of Maxwell-Cremona liftings, the stress coefficients defined above yield a self-stress.
We will prove the following theorem later in Section~\ref{Projective-proof}.

\begin{theorem}
  \label{thm:self-stress} 
  The projection-stress $\bar\omega_S$ is a self-stress of the framework
  $G_S(\bar p)$.
\end{theorem}

\begin{remark}
  Let us briefly discuss the asymptotic properties of this construction. Assume
  we have a trivalent vertex $p_1$ with adjacent edges $p_1p_2$, $p_1p_3$ and
  $p_1p_4$ (as in Definition~\ref{def:stresscoef}), an affine plane (proper
  projection plane) $\pi$ and a vector $\ell$ that is not parallel to $\pi$.
  Denote by $A_{ijk}$ the oriented area of the projection of the triangle
  $p_ip_jp_k$ to the plane $\pi$ along $\ell$. Denote also by $V_{1234}$ the
  determinant $\det(p_{2}-p_{1},p_{3}-p_{1},p_{4}-p_{1})$.

  Consider the shifted configuration 
  \begin{equation*}
    p_{i,t}= p_i + t\ell
    \qquad \hbox{and} \qquad
    \pi_t=\pi+t\ell.
  \end{equation*}
  Then we have
  \begin{align*}
    \bar\omega_{12,t} 
    &=
    \beta_{1,t}\beta_{2,t} \cdot  \omega_{12,t}(p_{1,t},p_{2,t};p_{3,t},p_{4,t})
    \\
    &=
    \beta_{1,t}\beta_{2,t} \cdot \frac{\det(p_{2,t}-p_{1,t},p_{3,t}-p_{1,t},p_{4,t}-p_{1,t})}{\det(p_{1,t},p_{2,t},p_{3,t})\det(p_{1,t},p_{2,t},p_{4,t})}\\
    &=
    \frac{V_{1234}}{(t\cos(\ell,\pi) A_{123})(t\cos(\ell,\pi)A_{124})}
    +O(t^{-2})
    \\
    &=
    \frac{V_{1234}}{A_{123}A_{124}}
    \cdot 
    \frac{1}{t^2\cos^2(\ell,\pi)}
    +O(t^{-2}),
  \end{align*}
  while $t\to\infty$.
  Here the leading coefficient 
  \begin{equation*}
    \frac{V_{1234}}{A_{123}A_{124}}
  \end{equation*}
  defines a remarkable self-stress for the framework of an affine projection of
  a polyhedral surface along $\ell$ to the plane $\pi$. This self-stress is
  proportional  to the self-stress of the corresponding Maxwell-Cremona
  lifting. 
\end{remark}

\section{Projective version of Maxwell-Cremona liftings in $\mathbb{R}^3$;
proof of Theorem~\ref{thm:self-stress}}
\label{Projective-proof}

In this section, we work in a natural projectivisation of classical
statics, which is convenient to study
tensegrities~\cite{karpenkov2021combinatorial}.
This projectivisation was studied by H.~Crapo and
W.~Whiteley~\cite{CW1982} and I.~Izmestiev~\cite{Izmestiev1,Izmestiev2}.

\subsection{Basics of projective statics}
\label{sec:projstatics}

Recall that the \emph{projective space} $\mathbb{R} P^n$ is the space of
one-dimensional subspaces of $\mathbb{R}^{n+1}$, 
which are denoted by homogeneous coordinates
\begin{equation*}
  \hat p=(a_1:\ldots:a_{n+1})\in \mathbb{R} P^n.
\end{equation*}
Denote by $\Lambda^{2}(\mathbb{R}^{n+1})$ the space of exterior 2-forms on
$\mathbb{R}^{n+1}$. We associate to each point
$p=(a_1,\ldots,a_{n+1})\in\mathbb{R}^{n+1}$ the 1-form
\begin{equation*}
  dp = a_1 dx_1 + \ldots + a_{n+1} dx_{n+1}.
\end{equation*}

\begin{definition}
  Let $\hat p, \hat q$ be two points in $\mathbb{R} P^n.$ 
  \emph{Forces} in $\mathbb{R} P^n$ acting at point $\hat p_i$ in direction
  $\hat q_i$ are represented by  decomposable
  2-forms $dp \wedge dq \in \Lambda^{2}(\mathbb{R}^{n+1})$.
\end{definition}

\begin{definition}
  Let $G$ be a graph. A \emph{projective framework} $G(\hat p)$ in $\mathbb{R}
  P^n$ consists of vertices $\hat p$ and edges $\hat p_i \hat p_j$ if $v_i
  v_j$ is an edge in $G$. (A projective framework is a realisation of a
  graph in projective space.)
\end{definition}

\begin{definition}
  Let $G(\hat p)$ be a projective framework and let $p_i \in \mathbb{R}^{n +
  1}$ be fixed representatives of vertices $\hat p_i$.
  A \emph{projective stress} $\omega$ is an assignment of a real number
  $\omega_{ij}$ to each non-oriented edge $\hat p_i \hat p_j$ (i.e.,
  $\omega_{ij} = \omega_{ji}$).
  We say that a projective stress is in \emph{equilibrium} or is a
  \emph{self-stress} if for every vertex $\hat p_i$ the 2-form of forces
  vanishes:
  \begin{equation*}
    \sum\limits_{\{j \mid ij \text{ is an edge}\}} \omega_{ij} dp_i \wedge
    dp_j = 0
  \end{equation*}
  (the summation is over all edges of $G(\hat p)$ adjacent to $\hat p_i$).
\end{definition}

\subsection{Projective liftings and corresponding self-stresses}
\label{sec:proofss}

In this subsection, we discuss self-stressability of projective frameworks
with projective stresses defined by projective liftings.

\begin{definition}
  Let $S(p)$ be a non-degenerate oriented polyhedral surface in $\mathbb{R}^3$
  and $G(\hat p)$ the corresponding projective framework in $\mathbb{R} P^2$.
  Let further $\omega_S$ be the projective lifting coefficients for $S(p)$
  as in Definition~\ref{oriented-def-stress}.
  We then say that $S(p)$ is a \emph{lifting} of the projective framework
  $G(\hat p)$ and the collection of 2-forms $\Psi_{ij} = \omega_{ij} dp_i
  \wedge dp_j$ is the projective stress \emph{associated} to the lifting $S(p)$. 
\end{definition}

\begin{theorem}
  \label{thm:projselfstress} 
  The projective stress $\Psi_S$ on $G(\hat p)$ associated to a lifting
  $S(p)$ is a self-stress.
\end{theorem} 
\begin{proof} 
  Let $S(p)$ be a non-degenerate oriented polyhedral surface in $\mathbb{R}^3$.
  Consider a $k$-valent vertex $p_0$ of $G(\hat p)$ with emanating edges
  $p_0p_1, \ldots, p_0p_k$. We prove the statement of the theorem by
  induction on $k$.
  We start with $k = 3$ and show $\sum_{i = 1}^3 \Psi_{0i} = 0$.
  We have
  \smallskip

  \noindent
  $\Psi_{01} + \Psi_{02} + \Psi_{03} 
  =
  \omega_{01} dp_0 \wedge dp_1 +
  \omega_{02} dp_0 \wedge dp_2 +
  \omega_{03} dp_0 \wedge dp_3
  =
  dp_0 \wedge d(\omega_{01} p_1 + \omega_{02} p_2 + \omega_{03} p_3)
  $

  $
  = dp_0 \wedge
  d\big(
  \frac{\det(p_1 - p_0, p_3 - p_0, p_2 - p_0)}
  {\det(p_0, p_1, p_3) \det(p_0, p_1, p_2)}
  p_1
  +
  \frac{\det(p_2 - p_0, p_1 - p_0, p_3 - p_0)}
  {\det(p_0, p_2, p_1) \det(p_0, p_2, p_3)}
  p_2
  +
  \frac{\det(p_3 - p_0, p_2 - p_0, p_1 - p_0)}
  {\det(p_0, p_3, p_2) \det(p_0, p_3, p_1)}
  p_3
  \big)
  $

  $
  = 
  \frac{\det(p_1 - p_0, p_2 - p_0, p_3 - p_0)}
  {\det(p_0, p_1, p_2) \det(p_0, p_2, p_3) \det(p_0, p_3, p_1)}
  dp_0 \wedge
  d\big(
  \det(p_0, p_2, p_3)
  p_1
  +
  \det(p_0, p_3, p_1)
  p_2
  +
  \det(p_0, p_1, p_2)
  p_3
  \big)
  $

  $
  = 
  \frac{\det(p_1 - p_0, p_2 - p_0, p_3 - p_0)}
  {\det(p_0, p_1, p_2) \det(p_0, p_2, p_3) \det(p_0, p_3, p_1)}
  dp_0 \wedge
  d\big(
  \det(p_1, p_2, p_3)
  p_0
  \big)
  =
  0.
  $
  \smallskip

  \noindent 
  Let us now assume that the statement holds for $k-1$. We show that it
  holds for $k$. 
  For that let $v$ be any non-zero vector in the intersection of the planes
  $p_0 p_1 p_2$ and $p_0 p_{k-1} p_k$. Let us set $q := p_0 + v$. Then, by
  Proposition~\ref{prop-does-not-depend}, we have
  \begin{equation*}
    \begin{array}{l@{\,}l@{\,}l@{\,}l}
      \omega_{01}
      &=
      \omega(p_0,p_1;p_k,p_2)
      &=
      \omega(p_0,p_1;p_k,q);
      \\
      \omega_{02}
      &=
      \omega(p_0,p_2;p_1,p_3)
      &=
      \omega(p_0,p_2;q,p_3);
      \\
      \omega_{0,k-1}
      &=
      \omega(p_0,p_{k-1};p_{k-2},p_k)
      &=
      \omega(p_0,p_{k-1};p_{k-2},q);
      \\
      \omega_{0k}
      &=
      \omega(p_0,p_{k};p_{k-1},p_1)
      &=
      \omega(p_0,p_k;q,p_1);
      \\
      &
      \hphantom{=}\;
      \omega(p_0,q;p_{k-1},p_2)
      &=
      \omega(p_0,q;p_k,p_1)
      =
      -\omega(p_0,q;p_1,p_k).
    \end{array}
  \end{equation*}
  Therefore,
  \begin{equation*}
    \sum\limits_{i=1}^k \Psi_{0i}
    =
    \sum\limits_{i=1}^k \omega_{0i} dp_0 \wedge dp_i
    =
    dp_0 \wedge d\big(\sum\limits_{i=1}^k \omega_{0i} p_i\big).
  \end{equation*}
  Denoting $\omega_{0q} := \omega(p_0, q; p_{k-1}, p_2) q$,
  the second part expands to 
  \smallskip

    $
    \sum\limits_{i=1}^k \omega_{0i} p_i
    =
    (
    \omega_{02} p_2
    +
    \ldots
    +
    \omega_{0,k-1} p_{k-1}
    +
    \omega_{0q} q
    )
    +
    (
    -
    \omega_{0q} q
    +
    \omega_{01} p_1
    +
    \omega_{0k} p_k
    )
    $

    $
    =
    \omega(p_0, p_2; p_1, p_3) p_2
    +
    \ldots
    +
    \omega(p_0, p_{k-1}; p_{k-2}, p_k) p_{k-1}
    +
    \omega(p_0, q; p_{k-1}, p_2) q 
    $

    $
    \hphantom{=}
    -
    \omega(p_0, q; p_{k-1}, p_2) q
    +
    \omega(p_0, p_1; p_k, p_2) p_1
    +
    \omega(p_0, p_k; p_{k-1}, p_1) p_k
    $

    $
    =
    \omega(p_0, p_2; q, p_3) p_2
    +
    \ldots
    +
    \omega(p_0, p_{k-1}; p_{k-2}, q) p_{k-1}
    +
    \omega(p_0, q; p_{k-1}, p_2) q 
    $

    $
    \hphantom{=}
    +
    \omega(p_0, q; p_1, p_k) q
    +
    \omega(p_0, p_1; p_k, q) p_1
    +
    \omega(p_0, p_k; q, p_1) p_k
    $

    $\overset{(*)}{=} 0 + 0 = 0.$
    \smallskip

    \noindent
  Equality $(*)$ holds since the first row vanishes by the induction
  hypothesis, while the second row vanishes by the base of induction ($k =
  3$). So we have the equilibrium at every vertex. This concludes the
  proof.
\end{proof}

\subsection{Proof of Theorem~\ref{thm:self-stress}}

We conclude this section with proving Theorem~\ref{thm:self-stress} which claims that the projected framework $G_S(\bar p)$ is a self-stressed framework with self-stress $\bar \omega_S$.

\begin{proof}[Proof of Theorem~\ref{thm:self-stress}]
  Let $S(p)$ be a polyhedral surface with the framework defined by its
  1-skeleton denoted by $G_S(p)$.
  Its projection on a proper projection plane $\pi$ yields the projective
  framework $G_S(\bar p)$ which satisfies the conditions of
  Definition~\ref{def-betas}. Let further $G_S(\hat p)$ denote the
  corresponding projective framework whose homogeneous coordinates $\hat p$ are
  given by the Euclidean coordinates $p$.
  Finally, let $\beta_i$ be the projection-coefficients of $S(p)$ to $\pi$.
  Recall that in this setting, the expression for the projection-stresses is as
  follows (Equation~\eqref{eq:projection-stress})
  \begin{equation*}
\bar\omega_{ij}=\beta_i\beta_j\omega_{ij}.
  \end{equation*}
  Consider a point $\bar p_0$ with its emanating edges 
  $\bar p_0\bar p_1,\ldots, \bar p_0\bar p_k$.
  Theorem~\ref{thm:projselfstress} implies that the corresponding projective
  stress $\Psi_S$ is a self stress hence we have equilibrium in each vertex
  star:
  \begin{align} 
    0 
    &=
    \sum\limits_{i=1}^{k} \omega_{0i} dp_0 \wedge dp_i
    \nonumber
    =
    \sum\limits_{i=1}^{k} 
    \omega_{0i} d(\beta_0 \bar p_0) \wedge d(\beta_i\bar p_i)
    \nonumber
    =
    \sum\limits_{i=1}^{k}
    \beta_0\beta_i\omega_{0i} d\bar p_0 \wedge d(\bar p_i-\bar p_0)
    \nonumber
    \\
    &=
    d\bar p_0
    \wedge 
    d\Big(\sum\limits_{i=1}^{k}
    \beta_0\beta_i\omega_{0i} (\bar p_i-\bar p_0)\Big) 
    \label{eq:wedgeproduct}
  \end{align} 
  The sum of all vectors in the $1$-from corresponding to the second component of
  the wedge product~\eqref{eq:wedgeproduct} is parallel to $\pi$, since all
  vectors $\bar p_i - \bar p_0$ are parallel to $\pi$. The vector corresponding
  to the $1$-form in the first component of~\eqref{eq:wedgeproduct} intersects
  $\pi$ and is hence not parallel to $\pi$. Therefore, the above equation is
  equivalent to 
  \begin{equation*}
    d\Big(\sum\limits_{i=1}^{k}
    \beta_0\beta_i\omega_{0i} (\bar p_i-\bar p_0)\Big) =0,
  \end{equation*}
  which is further equivalent to
  \begin{equation*}
    \sum\limits_{i=1}^{k} \beta_0\beta_i\omega_{0i} (\bar p_i-\bar p_0) =0,
  \end{equation*}
  and hence
  \begin{equation*}
    \sum\limits_{i=1}^{k}\bar\omega_{0i} (\bar p_i-\bar p_0)=0.
  \end{equation*}
  Thus the equilibrium condition is satisfied at every vertex of
  $G_S(\bar p)$, and so $\bar\omega_S$ is a self-stress of $G_S(\bar p)$.
\end{proof}

\section{Projective lifting coefficients in integer geometry}
\label{sec:projective-stress}

In this section we relate projection-stresses induced by integer surfaces to 
integer invariants of those surfaces.
We start in Section~\ref{subsec:integer-geometry} with definitions of integer invariants that we use in the formulae for the stresses. 
Further, in Section~\ref{subsec:how-to-compute} we write the analytic
expressions for integer invariants. Finally, in
Section~\ref{subsec:lifting-coefficients} we formulate and prove such formulae.

\subsection{Basic invariants of integer geometry}
\label{subsec:integer-geometry}

We call a point in $\mathbb{R}^n$ an \emph{integer point} if all its coordinates are
integers.
An \emph{integer segment} or \emph{integer vector} has integer points as endpoints
(see Figure~\ref{fig:integer-invariants} centre).
An \emph{integer plane} contains a full rank integer sub-lattice 
(see Figure~\ref{fig:integer-invariants} left).
An \emph{integer angle} between two intersecting integer planes (of any
dimension) contains a full-rank integer lattice.

Let us recall a well-known mathematical notion in group theory. 
The \emph{index} of a subgroup $H$ in a group $G$ is the number of left cosets
in $G$ with respect to $H$. 

\begin{definition}
  We define the following invariants in integer geometry. They  are all 
  non-negative integers:
  \begin{itemize}
    \item The \emph{integer length}  between two points $p$ and $q$ is the
      number of integer points in the interior of the segment $pq$ plus one. We
      denote it by $\Il(p,q)$ (see Figure~\ref{fig:integer-invariants}
      centre-top).
    \item  The \emph{integer area/volume} of the parallelogram/parallelepiped
      generated by integer vectors $v_1,\ldots,v_k$ is the index of the
      sublattice generated by these vectors in the integer sublattice of the
      affine space spanned by $v_1,\ldots,v_k$. We denote it by
      $\IV(v_1,\ldots,v_k)$. 
      (In the case of two vectors we denote it by $\IA(v_1,v_2)$.
      See Figure~\ref{fig:integer-invariants} right.)
    \item The \emph{integer distance} from a point $p$ to the plane $\pi$ (of
      any dimension) is the index of the sublattice generated by all integer
      vectors of $\{p\} \cup \pi$ in the integer lattice in $\Span(p,\pi)$. We
      denote it by $\Id (p,\pi)$ (see Figure~\ref{fig:integer-invariants}
      centre-bottom).
    \item The \emph{integer sine} of the integer angle between planes $\pi_1$
      and  $\pi_2$ is the index of the sublattice generated by all integer
      vectors of $\pi_1\cup \pi_2$ in the integer lattice of
      $\Span(\pi_1,\pi_2)$. We denote it by $\isin(\pi_1,\pi_2)$.
  \end{itemize}
\end{definition}

\begin{center}

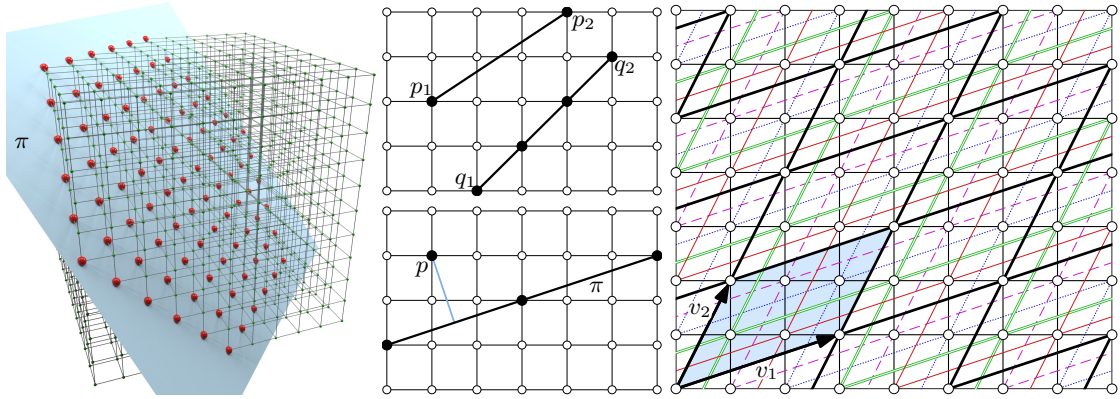
\begin{figure}[t]
  \begin{overpic}[width=.33\textwidth]{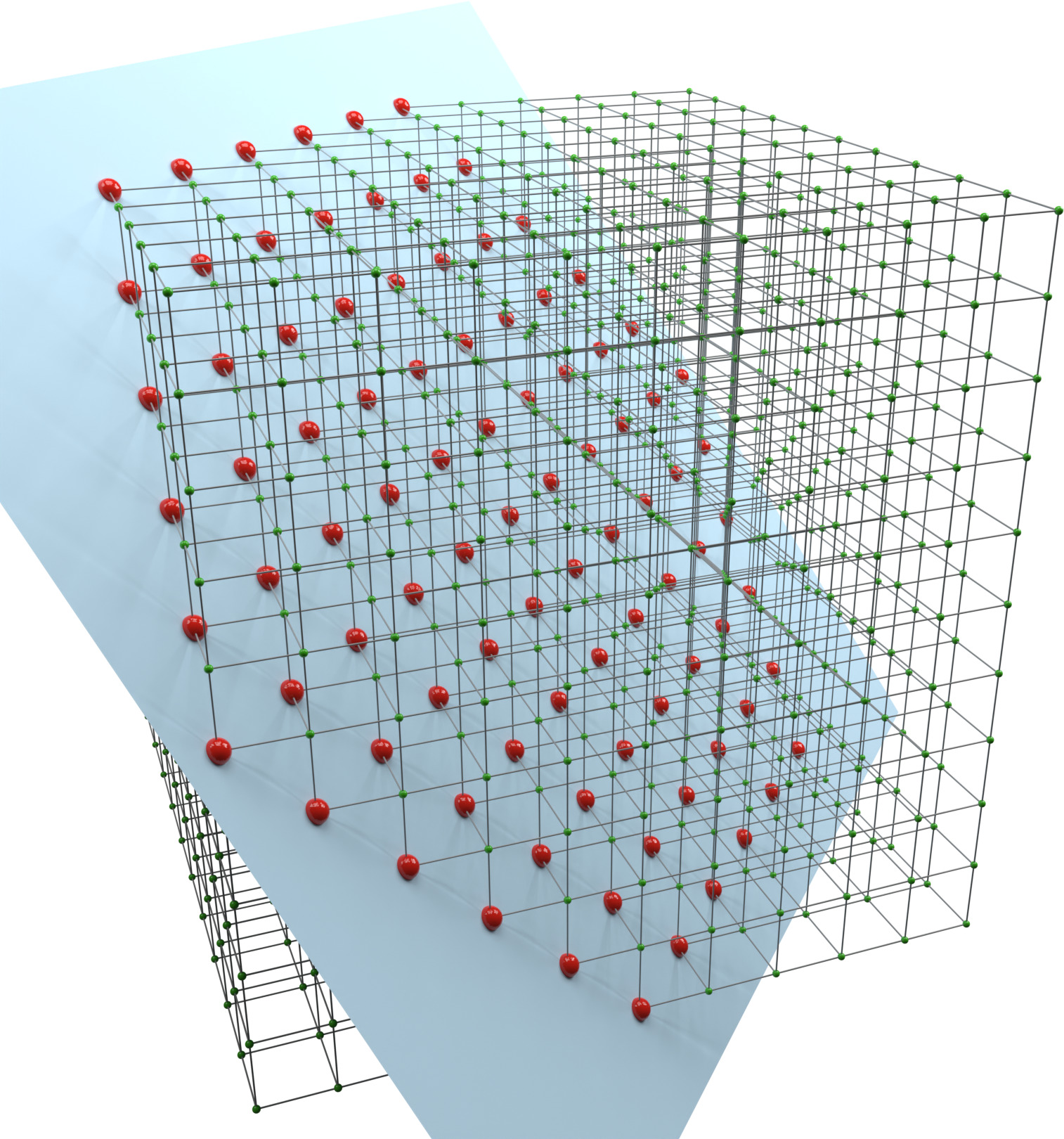}
    \put(2,63){\small$\pi$}
  \end{overpic}
  \begin{minipage}[b]{.25\textwidth}
    \begin{overpic}[width=\textwidth]{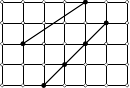}
      \put(10,37){\small$p_1$}
      \put(68,61){\small$p_2$}
      \put(26,4){\small$q_1$}
      \put(83,45){\small$q_2$}
    \end{overpic}
    \\[1.4mm]
    \begin{overpic}[width=\textwidth]{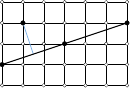}
      \put(12,44){\small$p$}
      \put(74,37){\small$\pi$}
    \end{overpic}
  \end{minipage}
  \begin{overpic}[width=.40\textwidth]{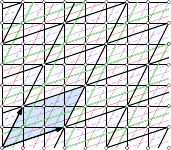}
    \put(19,4){\small\contour{white}{$v_1$}}
    \put(4,18){\small\contour{white}{$v_2$}}
  \end{overpic}
  \caption{\emph{Left:} Integer plane $\pi$ in $\ZZ^3$. It containes a full
  rank integer sub-lattice (bigger red dots).
  \emph{centre-top:} Integer length between points: $\Il(p_1, p_2) = 1$
  and $\Il(q_1, q_2) = 3$.
  \emph{Centre-bottom:} Integer distance from point $p$ to line $\pi$: $\Id(p,
  \pi) = 5$.
  \emph{Right:} Integer area of a parallelogram spanned by two vectors:
  $\IA(v_1, v_2) = 5$.
  The five cosets are illustrated in different colors and different
  line-styles.
  }
  \label{fig:integer-invariants}
\end{figure}
\end{center}

\subsection{How to compute integer invariants}
\label{subsec:how-to-compute}

The following properties of the above integer invariants are helpful for their
computations.

Let $v$ be an integer vector. Then we denote by $\gcd(v)$ the greatest common
divisor of its coordinates. For vectors $u,v$, we denote by $u\times v$ their
cross-product.

\begin{proposition}
  \label{prop:integer-formulae}
  Consider four distinct integer points $p_1, p_2, p_3, p_4 \in \ZZ^3$. Let
  $\pi_{123}$ and $\pi_{124}$ be the planes spanned by 
  $p_1 p_3 p_2$ and $p_1 p_4 p_2$, respectively. Then the following holds:  
  \begin{enumerate}
    \item\label{enum1} integer length: 
      $\Il(p_1, p_2) = \gcd(p_2 - p_1)$;
    \item\label{enum2} integer area: 
      $\IA(p_3 - p_1, p_4 - p_1) = \gcd((p_3 - p_1) \times (p_4 - p_1))$;
    \item\label{enum3} integer volume: 
      $\IV(p_3 - p_1, p_4 - p_1, p_2 - p_1)
      = |\det(p_3 - p_1, p_4 - p_1, p_2 - p_1)|$;
    \item\label{enum4} integer distance to a line: 
      $\displaystyle \Id(p_3 , \Span(p_1, p_2))
      = \frac{\IA(p_1 - p_3, p_2 - p_3)}{\Il(p_1, p_2)};
      $
    \item\label{enum5} integer distance to a plane: 
      $\displaystyle \Id(p_4, \pi_{123})
      = \frac{\IV(p_1 - p_4, p_2 - p_4, p_3 - p_4)}{\IA(p_2 - p_1, p_3 - p_1)}$;
    \item\label{enum6} integer sine: 
      $\displaystyle\isin(\pi_{123}, \pi_{124})
      = \frac{\IV(p_3 - p_1, p_4 - p_1, p_2 - p_1)}{\Il(p_1, p_2) 
      \Id(p_3, \Span(p_1, p_2)) \Id(p_4, \Span(p_1, p_2))}$. 
  \end{enumerate}
\end{proposition}
\begin{proof}
  Properties \ref{enum1}, \ref{enum2} and \ref{enum3} follow immediately from
  the definition and \cite[Th.~18.30]{karpenkov-book}.

  Properties \ref{enum4}, \ref{enum5} and \ref{enum6} are relations between
  integer lattice invariants, and therefore they can be computed for a simple
  choice of integer lattice coordinates:
  \begin{equation*}
    p_1=(0,0,0),
    \qquad 
    p_2=(a,0,0),
    \qquad
    p_3=(b,c,0),
    \qquad
    p_4=(q,r,s),
  \end{equation*}
  for some positive integers $a,b,c,q,r,s$.
  In this case we obtain:
  \begin{equation*}
    \begin{array}{lll}
      \Il(p_1, p_2) = a,
      \quad &
      \IA(p_1 - p_3, p_2 - p_3) = a c,
      \quad &
      \IV(p_3 - p_1, p_4 - p_1, p_2 - p_1) = a c s,
      \\
      \Id(p_3, \Span(p_1, p_2)) = c,
      \quad &
      \IA(p_2 - p_1, p_3 - p_1) =  a c,
      \quad &
      \isin(\pi_{123}, \pi_{124}) = s/\gcd(r, s),
      \\
      \Id(p_4, \pi_{123}) = s,
      \quad &
      \Id(p_4, \Span(p_1, p_2)) = \gcd(r,s),
    \end{array}
  \end{equation*}
  Now the formulae are straightforward.
\end{proof}

\subsection{Projective lifting coefficients of integer surfaces}
\label{subsec:lifting-coefficients}

Let us now consider four points $p_1, p_2, p_3, p_4$ with integer coordinates.
Then we can express the projective lifting coefficient $\omega_{12}$ in terms
of integer invariants. Consequently, $\omega_{12}$ is a rational number. Note
that since all integer invariants are non-negative integers, we must take the
absolute value of $\omega$.

\begin{proposition}
  \label{integer-propos}
  Assume that $p_1, p_2 \in \ZZ ^3$ and let $\pi_1, \pi_2$ be two integer
  planes containing the line through $p_1 p_2$. Then the projective lifting
  coefficient for the edge $p_1 p_2$ induced by the planes $\pi_1$ and $\pi_2$
  satisfies:
  \begin{equation*}
    |\omega(p_1,p_2;p_3,p_4)|=
    \frac{\isin(\pi_1,\pi_2)}
    {\Il(p_1,p_2)
    \Id(O,\pi_1)
    \Id(O,\pi_2)},
  \end{equation*}
  where as before $p_3\in \pi_1\setminus \pi_2$ and $p_4\in \pi_2\setminus
  \pi_1$ can be chosen arbitrarily
  $($in brief we write $\omega_{12}=\omega(p_1,p_2;p_3,p_4)$$)$.
\end{proposition}
\begin{proof}
  The formula can be derived by directly computing both the right hand side and
  the left hand side of the equation. 
  First of all, note that both sides do not depend on the choice of a basis in
  $\ZZ^3$. Secondly, the
  points $p_3$ and $p_4$ in the computations can be chosen away from
  $\pi_1\cap\pi_2$ arbitrarily in the planes $\pi_1$ and $\pi_2$, respectively.

  For that reason, without loss of generality, the plane $\pi_1$ can be chosen
  to be $x=n$ for some integer $n$.
  In addition, we can assume that $p_1=(n,0,0)$, $p_2=(n,a,0)$ for some
  positive integer $a$, and $p_3=(n,0,1)$. Further, we pick the plane $\pi_2$
  passing through $p_1p_2$ as above and through $p_4=(n+b,0,c)$ for some
  relatively prime positive integers $b$ and $c$. 
  Now we compute the expressions of the equation for that choice of coordinates
  with the help of the identities in Proposition~\ref{prop:integer-formulae}
  and obtain
  \begin{equation*}
    \begin{array}{c}
      \displaystyle
      \isin(\pi_1,\pi_2)=b, 
      \qquad
      \Il(p_1p_2)=a,
      \qquad
      \Id(O,\pi_1)=\frac{an}{a}=n,
      \qquad
      \Id(O,\pi_2)=\frac{acn}{a}=cn;
      \\
      \displaystyle
      \omega_{S}(p_1p_2)=\frac{b}{acn^2}.
    \end{array}
  \end{equation*}
  This concludes the proof.
\end{proof}

\begin{remark}
  Later in Section~\ref{sec:self-stresses-klein-sails}, we will deal with Klein
  sails. They are convex polyhedral surfaces. As a result, we consider only
  convex surfaces which, subsequently, are corresponding to self-stresses of
  the same sign, so the computation of signs of $\omega_{S}$ does not play any
  role (since both $\pm\omega_{S}$ are self-stresses).
\end{remark}

\begin{example}
Consider the following four points
\begin{equation*}
p_1=(2,3,4),
\qquad
p_2=(2,7,9),
\qquad
p_3=(1,-3,5),
\qquad
p_4=(5,6,7).
\end{equation*}
Then, on the one hand, by Definition~\ref{def:stresscoef}
we have
\begin{equation*}
\omega(p_1,p_2;p_3,p_4)=\frac{\det(p_2-p_1,p_3-p_1,p_4-p_1)}{\det(p_1,p_2,p_3)\det(p_1,p_2,p_4)}
=\frac{99}{69\cdot (-9)}=
-\frac{11}{69}.
\end{equation*}
On the other hand, by Proposition~\ref{integer-propos} we get
\begin{equation*}
|\omega(p_1,p_2;p_3,p_4)|=
\frac{\isin(\pi_1,\pi_2)}
{\Il(p_1,p_2)
\Id(O,\pi_1)
\Id(O,\pi_2)}=
\frac{33}{1\cdot 69\cdot 3}=
\frac{11}{69}.
\end{equation*}
\end{example}

\section{Self-stresses corresponding to periodic Klein sails}
\label{sec:self-stresses-klein-sails}

Let us show how to use multidimensional continued fractions in the sense of Klein in order to generate algebraically periodic tensegrities. 
We start in Section~\ref{subsec:multi-dimensional} with general  definitions of multidimensional continued fraction and corresponding sails. Further, in Section~\ref{subsec:periodic-sails}, we discuss periodicity of algebraic sails. In Section~\ref{subsec:algebraically-periodic}, we apply the theory of self-stress projections to Klein sails.
We go through several 
examples in
Section~\ref{subsec:examples-main}. Finally, in
Section~\ref{subsec:algorithmic-question} we address the algorithmic questions regarding the computations of
sails, and we show how to reconstruct (if possible) a surface from a
self-stressed framework in an affine plane not containing the origin.

\subsection{Multi-dimensional sails and continued fractions in the sense of
Klein}
\label{subsec:multi-dimensional}

Consider $n$ linearly independent vectors $x_1, \ldots, x_n \in \mathbb{R}^n$
and let $Ox_i = \{\lambda x_i \mid \lambda \geq 0\}$ denote the \emph{ray}
through $x_i$ and starting at $O$. A \emph{simplicial cone} $\mathcal{C}(x_1,
\ldots, x_n)$ in $\mathbb{R}^n$ is the convex hull of the rays $Ox_1, \ldots , Ox_n$.

\begin{definition}
  Consider a simplicial cone $\mathcal{C}$. The convex hull of all integer
  points (in the interior or on the boundary) of
  $\mathcal{C}$ except the origin, is called the \emph{A-hull}. The boundary of
  the A-hull of $\mathcal{C}$ is said to be the \emph{sail} of the cone (in the
  sense of Klein).
\end{definition}

\begin{center}
\begin{figure}[h]
  \begin{overpic}[width=.4\textwidth]{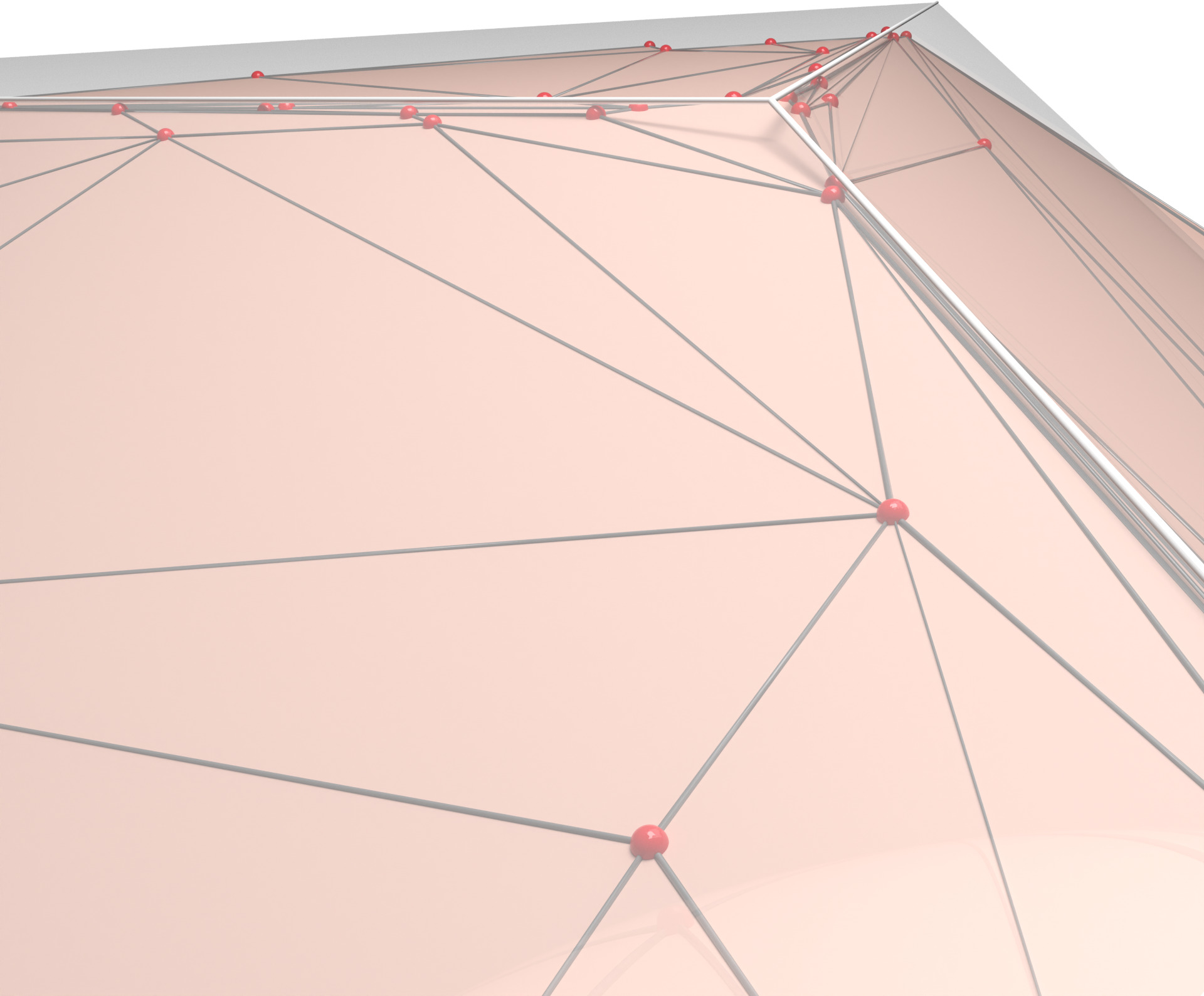}
    \put(79,82){\small$x_1$}
    \put(60,70){\small$O$}
  \end{overpic}
  \caption{An illustration of a sail. The simplicial cone
  $\mathcal{C}(x_1, x_2, x_3)$ has three faces and three edges (illustrated in
  white). The vertices of $\ZZ^3$ on the A-hull, the convex hull of 
  $\ZZ^3 \cap \mathcal{C}\setminus\{O\}$, are illustrated in red.}
\end{figure}
\end{center}

In the case where all $x_i$'s are rational vectors, their convex hull is a polyhedral surface with finitely many facets. In other cases, the boundary is unbounded and has infinitely many facets.

\begin{definition}
  Consider a matrix $A\in \GL(n,\mathbb{R})$ with distinct positive real eigenvalues.
  Then $A$ has $n$ distinct invariant one-dimensional eigenspaces $\Span(x_1),
  \ldots, \Span(x_n)$. Each eigenspace defines two rays $Ox_i$ and $O(-x_i)$.
  Taking every combination of them defines $2^n$ simplicial cones.
  The collections of the sails of these cones is said to be the \emph{continued
  fraction} for $A$ (in the sense of Klein).
\end{definition}

For the general theory of multidimensional continued fractions in the sense of Klein, we refer the reader to the books~\cite{Arn02} and~\cite{karpenkov-book}.

\subsection{Periodic sails}
\label{subsec:periodic-sails}

Let us describe a 
class of simplicial cones and their algebraic
sails.

\begin{definition}
  In the case of $A\in \GL(n,\ZZ)$, the corresponding sails and continued
  fractions are called \emph{algebraic}. 
\end{definition}

The centralizer of $A$ in $\GL(n,\ZZ)$ is called the \emph{Dirichlet group} of
$A$ denoted by $\Xi(A)$. The positive Dirichlet group $\Xi_+(A)$ is the
multiplicative subgroup of $\Xi(A)$ consisting of all matrices whose
eigenvalues are positive real numbers. By the Dirichlet's unit theorem,
$\Xi_+(A)$ is isomorphic to $\ZZ^{n-1}$. This group preserves the
sails forming the continued fraction for $A$ and acts transitively on it. The
quotient of such a sail with respect to this group is
an $(n{-}1)$-dimensional torus. As a consequence, the combinatorial structure
of such an algebraic sail is $(n{-}1)$-periodic.
So we arrive at the following surprising result (for further discussions and
proofs see~\cite{karpenkov-book,BS66}).

\begin{theorem}
  Algebraic sails have a doubly-periodic combinatorial structure.
  In particular, their 1-skeletons are doubly-periodic.
\end{theorem}

\subsection{Klein-Arnold self-stressed frameworks and their periodicity in the
algebraic case}
\label{subsec:algebraically-periodic}

Let us define self-stresses obtained from projections of Klein sails. 

\begin{definition}
  Let $S(p)$ be a sail in $\mathbb{R}^3$ with vertices $p_i$. Consider a proper
  projection plane $\pi$ for $S(p)$. The self-stressed framework
  $G_S(\bar p)$ obtained as a projection of the sail to the plane $\pi$ is
  called the \emph{Klein-Arnold self-stressed framework} corresponding to
  $S(p)$.
  If $S(p)$ is an algebraic sail, we say that the corresponding self-stressed
  framework is \emph{algebraic}. 
\end{definition}

It turns out that algebraic Klein-Arnold self-stressed frameworks have a
certain type of periodicity, which we discuss in Theorem~\ref{thm:stresscoefprop}
below.

\begin{definition} 
  A point is \emph{rational} if all its coordinates are rational numbers.
  A framework is \emph{rational} if all its vertices are rational points.
  A self-stress is called \emph{rational/integer} if all its stress
  coefficients are rational/integer.
\end{definition}

\begin{remark}
  Note that not every rational infinite self-stress is proportional to an
  integer self-stress. For instance, let us consider a simple one-dimensional
  tensegrity with vertices $p_i = 2^i$ with $i \in \ZZ$ and edges $p_i p_{i + 1}$ connecting
  two consecutive vertices. Let stresses be given by
  $\omega_{i, i + 1} = 2^{-i}$. Then $\omega_{i, i + 1} (p_{i + 1} - p_i) = 1$
  for all $i$, hence $\omega$ is a self-stress. This self-stress is not integer
  as its stress coefficients are (completely reduced) rational numbers with
  denominators that are unbounded from above. 
\end{remark}

\subsubsection{Periodicity of algebraic tensegrities}

\begin{theorem} 
  \label{thm:stresscoefprop}
  Let $S(p)$ be an algebraic sail of some simplicial cone in $\mathbb{R}^3$, and let
  $\pi$ be a proper projection plane for $S(p)$.
  \begin{enumerate}
    \item\label{itm:sail-i} 
      The projective lifting coefficients of $\omega_S$ are rational numbers
      that simultaneously have the same sign.
    \item\label{itm:sail-ii}
      The projective lifting coefficients of $\omega_S$ are doubly periodic.
    \item\label{itm:sail-iii} 
      If $\pi$ is an integer plane, then the projection-framework $G_S(\bar p)$
      is rational and $\bar \omega_S$ is proportional to some integer
      self-stress.
  \end{enumerate}
\end{theorem}
\begin{proof}
  We first consider \ref{itm:sail-i}. All points of the sail are integer, so
  all the determinants in Definition~\ref{def:stresscoef} are integer as well.
  Therefore the projective lifting coefficients are rational.
  The signs of $\omega_S$ are consistent due to the convexity of the sail and
  its fixed orientation as defined. 

  As for \ref{itm:sail-ii}, note that the value of determinants is invariant
  under multiplication by elements of the Dirichlet group, which consists of
  $\SL(3,\ZZ)$ matrices, all of which have determinant one. 
  By Definition~\ref{def:stresscoef}, the projective lifting coefficients are
  ratios of three determinants, making them invariant under the action of the
  Dirichlet group.
  Since the corresponding positive Dirichlet group is isomporphic
  to $\ZZ^2$ we obtain double periodicity.

  Property~\ref{itm:sail-iii} follows from~\ref{itm:sail-ii} together with the fact
  that all the projection-coefficients $\beta_i$ are rational. 
\end{proof}

\subsection{Examples of periodic Klein-Arnold self-stressed frameworks}
\label{subsec:examples-main}

We start with the simplest example of two-dimensional continued fraction
which can be seen as a generalisation of the ``golden ratio'' to three-dimensions. It was independently first found by E.~Korkina~\cite{Kor95},
G.~Lachaud~\cite{lachaud}, A.D.~Bryuno and V.I.~Parusnikov~\cite{bryuno+}.
(It is also listed in~\cite[Ex.~22.9, dim $3$]{karpenkov-book}.) 

\begin{center}

\begin{figure}[t]
  \begin{minipage}{.57\textwidth}
    \begin{overpic}[width=\textwidth]{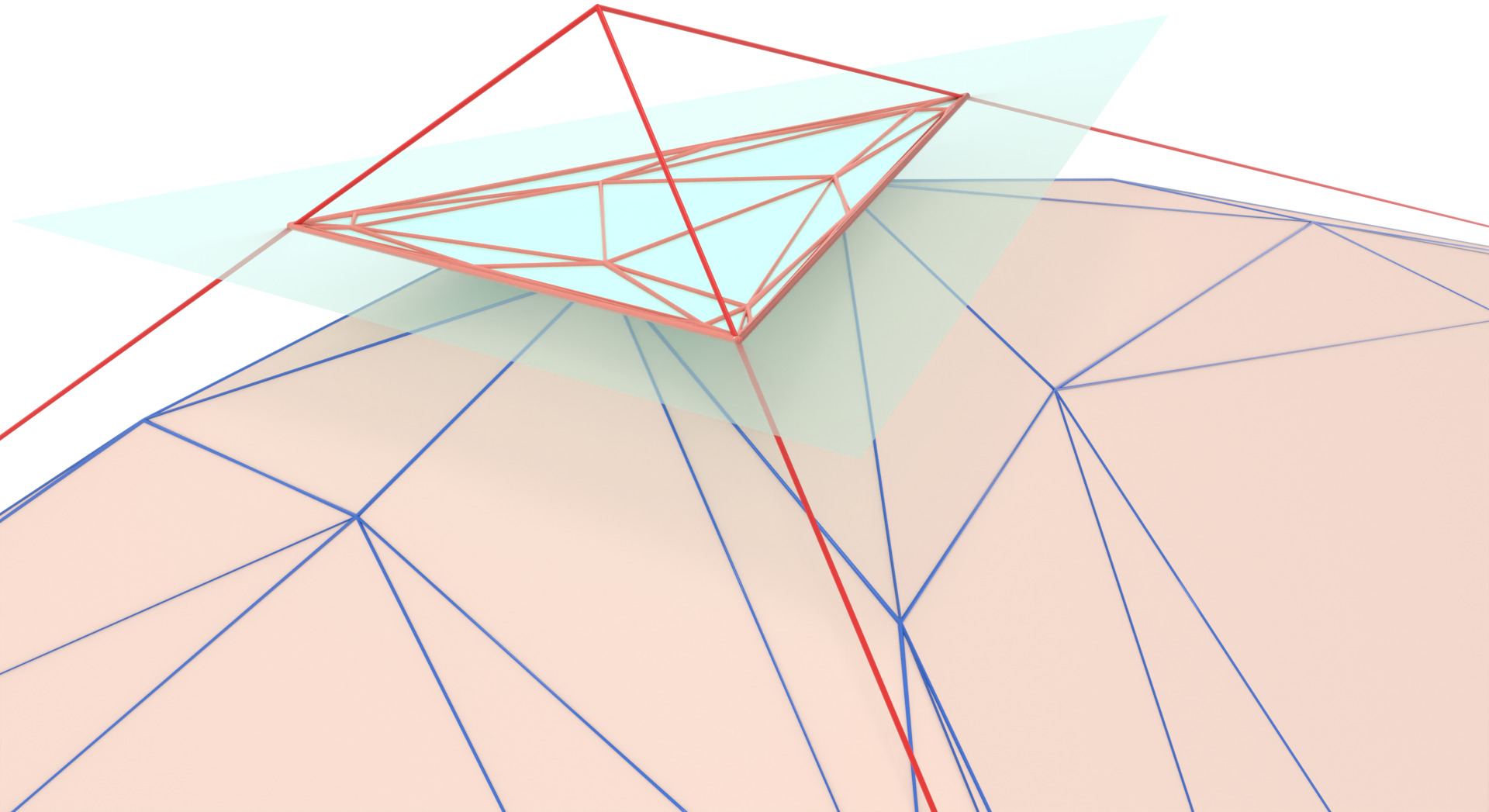}
      \put(40,55){\small$O$}
      \put(0,28){\small\rotatebox{38}{$\Span(x_1)$}}
      \put(58,13){\small\rotatebox{-63}{\contour{white}{$\Span(x_2)$}}}
      \put(85,44){\small\rotatebox{-13}{$\Span(x_3)$}}
      \put(2,3){\small\contour{white}{$G_S(p)$}}
      \put(71,50){\small\contour{white}{$\pi_1$}}
      \put(32,45){\small$G_S(\bar p)$}
    \end{overpic}
    \centerline{
    \begin{overpic}[width=.8\textwidth]{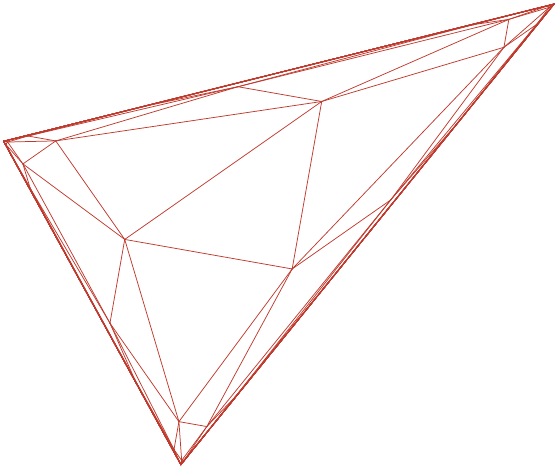}
      \put(32,45){\small$G_S(\bar p)$}
    \end{overpic}}
  \end{minipage}
  \hfill
  \begin{minipage}{.33\textwidth}
    \begin{overpic}[width=\textwidth]{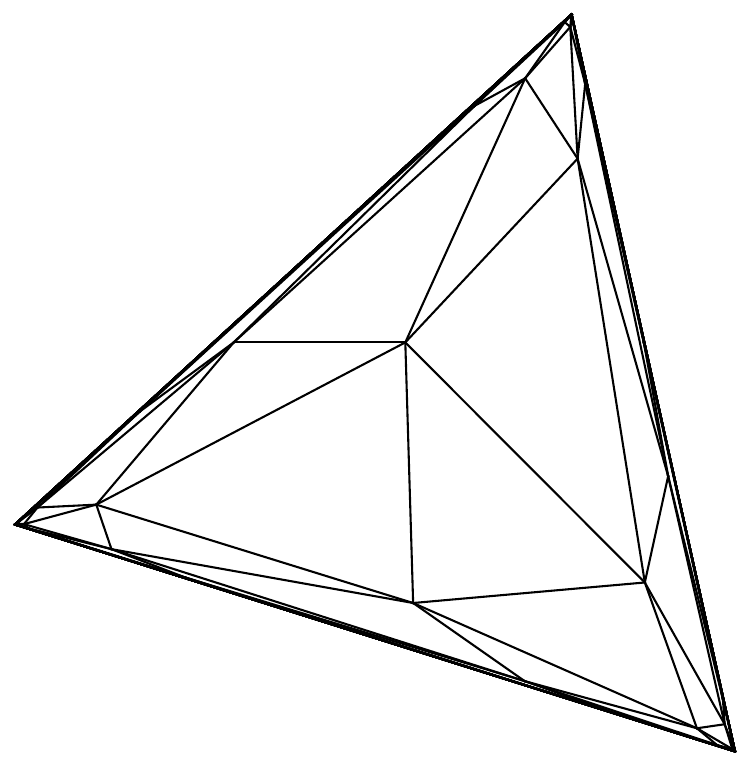}
    \end{overpic}
    \begin{overpic}[width=\textwidth]{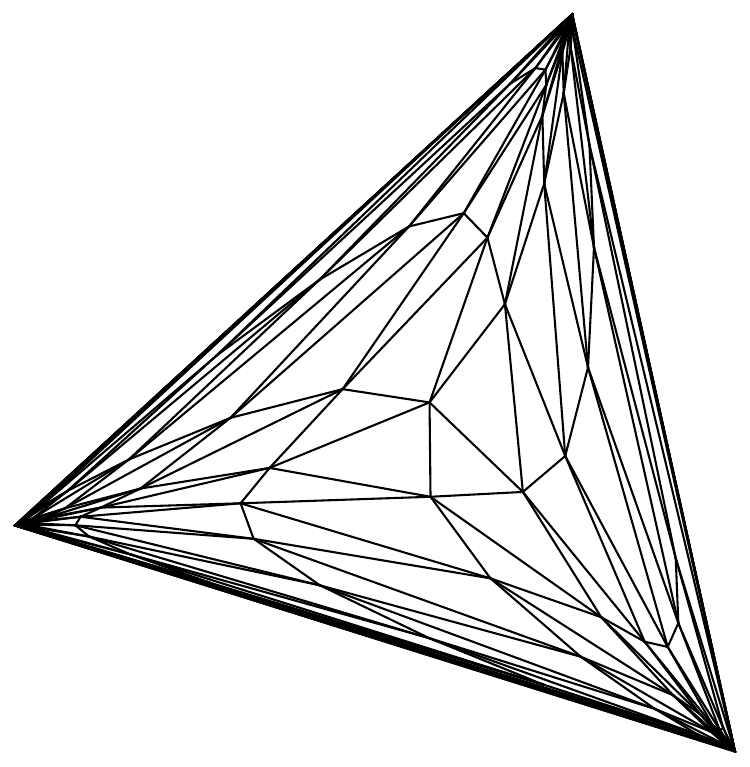}
    \end{overpic}
  \end{minipage}
  \caption{A sail $S(p)$ and its projection $G_S(\bar p)$ on the proper
  projection plane $\pi_1$ given by equation $z = 1$ \emph{(top-left)}. The
  centre of projection is $O = (0,0,0)$ and the sail is enclosed within
  the cone generated by the eigenspaces $\Span(x_1), \Span(x_2),
  \Span(x_3)$. 
  The resulting framework $G_S(\bar p)$ without projective distortion is
  depicted at \emph{bottom-left}. The triangle shape of the boundary is the
  intersection of the corresponding simplicial cone with $\pi_1$. Most edges of
  the framework $G_S(\bar p)$ accumulate along these triangle edges.
  Projecting the framework to another plane $\pi_2$, which is given by the
  normal vector $x_1/\|x_1\| + x_2/\|x_2\| + x_3/\|x_3\|$ (average of
  eigenvectors), is depicted at \emph{top-right}. Projecting to the same
  plane $\pi_2$ but from centre $(0,0,-2)$ results in the framework
  at \emph{bottom-right}. Here more of the edges that are accumulating along
  the triangle edges are visible.
  } 
\label{fig:sail-framework}
\end{figure}
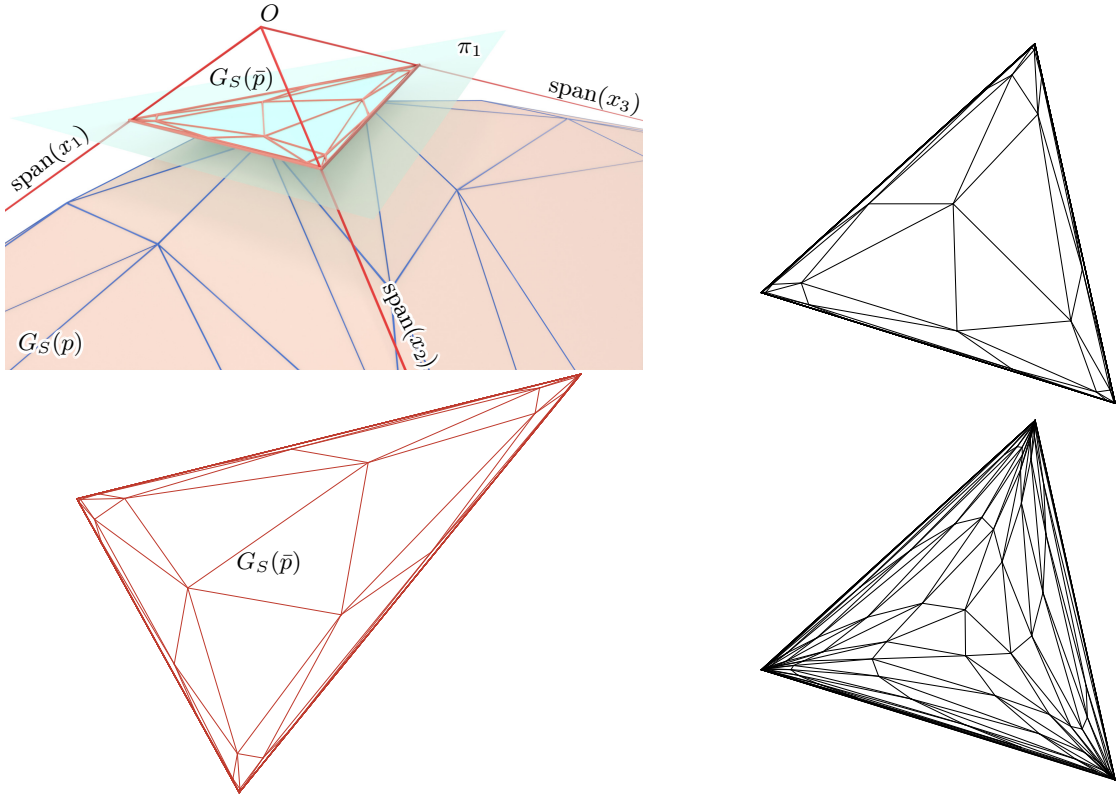
\end{center}

\begin{example}
  \label{example1}
  Consider the matrix 
  $
    L =
    \begin{pmatrix}
      1 & 1 & 1 \\
      1 & 2 & 2 \\
      1 & 2 & 3
    \end{pmatrix}. 
  $
  Its positive Dirichlet group is isomorphic to $\ZZ^2$ and it is generated by
  the following two matrices:
  \[
    M=L^{-1}\cdot(L-I)^2 \quad \text{and}\quad N=L.
  \]
  Note that all the vertices of the sail are of the form
  $p_{ij}=M^iN^j(0,0,1)^\top$.
  The sail and its projection to the plane $z=1$ are shown in
  Figure~\ref{fig:sail-framework} (left).
  The projection of the sail to the plane with the normal vector formed by
  the arithmetic mean of the normalized eigenvectors of $L$ is depicted in 
  Figure~\ref{fig:sail-framework} (right) together with a projection from
  a different centre. 
  Note that the stress coefficients computed below refer to the
  framework projected to plane $z=1$.

Let us describe in more details the  
associated fundamental domain of one of the sails with respect to the action of the Dirichlet group $\Xi(M)$. As shown in Figure~\ref{fig:stresses} (centre), it has 
\begin{itemize}
  \item one vertex $p_{00}$;
  \item three edges: $p_{01}p_{00}$, $p_{01}p_{10}$, and $p_{01}p_{11}$;
  \item two triangular faces: $p_{01}p_{00}p_{10}$ and $p_{01}p_{10}p_{11}$.
\end{itemize}

\begin{center}

\begin{figure}[t]
  \begin{overpic}[width=1\textwidth]{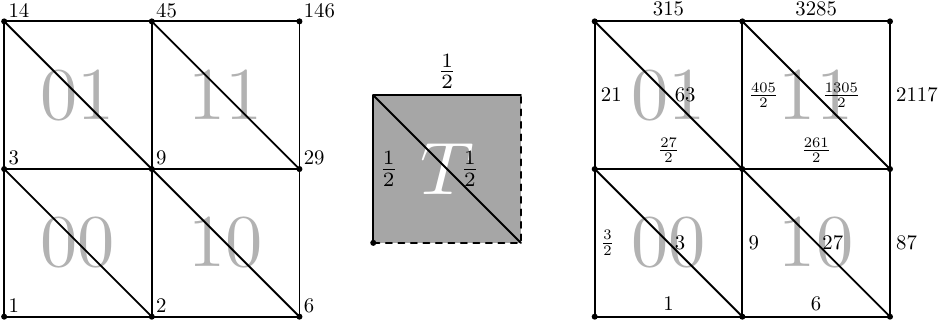}
    \put(38,6){$p_{00}$}
    \put(54,6){$p_{10}$}
    \put(38,25){$p_{01}$}
    \put(54,25){$p_{11}$}
  \end{overpic}
  \caption{Combinatorial view of Example~\ref{example1} (the actual
  geometry is depicted in Figure~\ref{fig:sail-framework}). The
  projection-coefficients $\beta_i$ are on the \emph{left}; the projective
  lifting coefficients $\omega_{ij}$ for a fundamental domain are in the
  \emph{centre}; the projection-stresses $\bar\omega_{ij}$ are on the
  \emph{right}.}
  \label{fig:stresses}
\end{figure}
\end{center}

Using the formulae of Definition~\ref{def:stresscoef} we get
\begin{equation*}
  \omega_{01,00}=\frac{1}{2},
  \qquad
  \omega_{01,10}=\frac{1}{2},
  \qquad
  \omega_{01,11}=\frac{1}{2}.
\end{equation*}
Finally, the projection-coefficient $\beta_{ij}$ is equal to the third
coordinate of $p_{ij}$. Then the projection-stresses for $\bar\omega$ are 
\begin{equation*}
  \begin{array}{l}
    \displaystyle
    \bar\omega_{i(j+1),ij}=\frac{1}{2}\beta_{i(j+1)}\beta_{ij};
    \\
    \displaystyle
    \bar\omega_{i(j+1),(i+1)j}=\frac{1}{2}\beta_{i(j+1)}\beta_{(i+1)j};\\
    \displaystyle
    \bar\omega_{i(j+1),(i+1)(j+1)}=\frac{1}{2}\beta_{i(j+1)}\beta_{(i+1)(j+1)}.
\end{array}
\end{equation*}
For small $i,j$ we have the values shown in Figure~\ref{fig:stresses}. In the centre part we show (combinatorially) a fundamental domain of the tensegrity. We write down the projective lifting coefficients $\omega_{ij}$ for the faces.
On the left and right sides of the figure, we further combinatorially illustrate several periods of the tensegrity corresponding to shifts by $\Id$, $M$, $N$, and $MN$. The exponents of these periods are indicated by large grey pairs of numbers. The left part displays the projection-coefficients
$\beta_i$, while the right part shows the projection stresses.
\end{example}

\begin{example}\label{example2}
  In this example we continue with a more complicated fundamental domain
  containing one pentagon and several triangles as faces
  (Parusnikov~\cite{Par95}, also listed in~\cite[Ex.~22.12]{karpenkov-book}). 
  Consider 
  the matrix 
  \[
    L =
    \begin{pmatrix}
      0 & 1 & \hphantom{-}0 \\
      0 & 0 & \hphantom{-}1 \\
      1 &  1 & -3
    \end{pmatrix}. 
  \] 
  Its positive Dirichlet group is isomorphic to $\ZZ^2$ and it is generated by
  the following two matrices:
  \[
    M=L^{-2} \quad \text{and}\quad N=M(3\Id-2L^{-1}).
  \]
  The projection of the sail to the plane $y=1$ is shown in
  Figure~\ref{fig:sail-framework2}.
Consider again $p_{ij,k}=M^iN^jp_{00,k}$  
for $k=1,2,3$, where
\begin{equation*}
  p_{00,1}=
  \begin{pmatrix}
    \hphantom{-}0\\
    \hphantom{-}2\\
    -5
  \end{pmatrix},
  \qquad
  p_{00,2}=
  \begin{pmatrix}
    \hphantom{-}0\\
    \hphantom{-}1\\
    -2
  \end{pmatrix},
  \quad
  \hbox{and}
  \quad
  p_{00,3}=
  \begin{pmatrix}
    -1\\
    \hphantom{-}1\\
    -1
  \end{pmatrix}.
\end{equation*}
Then, as depicted in Figure~\ref{fig:stresses2.pdf} (centre), the fundamental domain of one of the sails is as follows:
\begin{itemize}
  \item three vertices $p_{00,1}$, $p_{00,2}$, $p_{00,3}$;
  \item seven edges:\ 
    $e_1 = p_{00,1}p_{10,1}$, \
    $e_2 = p_{00,2}p_{10,1}$, \
    $e_3 = p_{10,2}p_{10,1}$, \
    $e_4 = p_{10,2}p_{11,1}$,
    \\
    \noindent
    \hphantom{seven edges:}
    $e_5 = p_{11,1}p_{00,3}$, \
    $e_6 = p_{10,2}p_{00,3}$, \
    $e_7 = p_{01,1}p_{00,3}$;
  \item three triangular faces:
    $p_{00,1}p_{10,1}p_{00,2}$, 
    $p_{01,1}p_{00,3}p_{11,1}$,
    and $p_{00,3}p_{10,2}p_{11,1}$;
    \\
    one pentagonal face:
    $p_{10,1}p_{10,2}p_{00,3}p_{01,1}p_{00,2}$.
\end{itemize}

\begin{centering}
\begin{figure}[t]
\includegraphics[width=.25\textwidth]{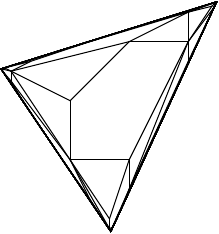}
\caption{The projection-framework $G_S(\tilde p)$ in Example~\ref{example2}.} 
\label{fig:sail-framework2}
\end{figure}
\end{centering}

\begin{center}
\begin{figure}[t]
  \begin{overpic}[width=1\textwidth]{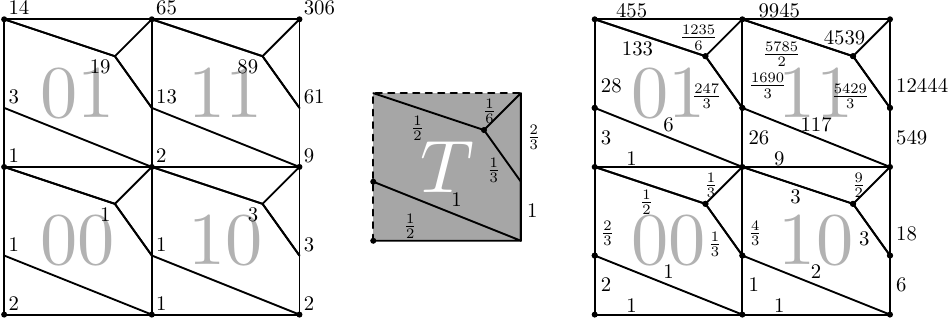}
    \put(37,6){$p_{00,1}$}
    \put(53,6){$p_{10,1}$}
    \put(37,25){$p_{01,1}$}
    \put(53,25){$p_{11,1}$}
    \put(34,14){$p_{00,2}$}
    \put(55.5,14){$p_{10,2}$}
    \put(47,18.5){\small$p_{00,3}$}
  \end{overpic}
  \caption{Example~\ref{example2}: the projection-coefficients $\beta_i$ are on the left; the projective lifting coefficients $\omega_{ij}$ for a fundamental domain are in the centre;
  the projection-stresses $\bar\omega_{ij}$ are on the right.}
  \label{fig:stresses2.pdf}
\end{figure}
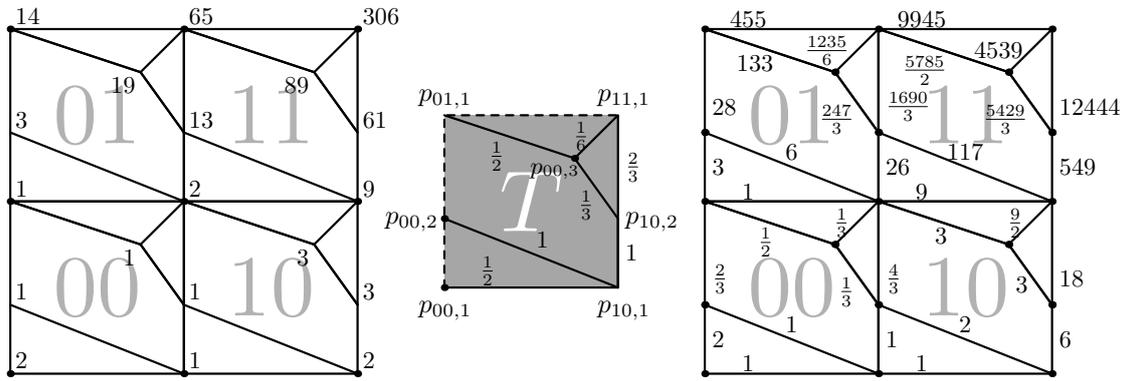
\end{center}
Using the formulae of Definition~\ref{def:stresscoef}
we obtain
\begin{equation*}
\omega_1=\frac{1}{2},
\qquad
\omega_2=1,
\qquad
\omega_3=1,
\qquad
\omega_4=\frac{2}{3},
\qquad
\omega_5=\frac{1}{6},
\qquad
\omega_6=\frac{1}{3},
\qquad
\omega_7=\frac{1}{2}.
\end{equation*}
Finally, the value of $\bar \omega$ for any edge of the graph is obtained by multiplication of the corresponding $\omega_i$ with the $y$-coordinates of both ends of this edge (see Figure~\ref{fig:stresses2.pdf}).
\end{example}

\begin{example}
  As we showed in~\cite{Kar04}, the examples of such sails form entire
  families. Let us consider one such family (Korkina~\cite{Kor95}, also
  listed in~\cite[Example 22.14, case $a\ge 0$]{karpenkov-book}).
  Consider a positive integer $a \geq 0$ and the matrix 
  \[
    M_{a} =
    \begin{pmatrix}
      0 & 0 & 1 \\
      1 & 0 & -a-5 \\
      0 &  1 & \hphantom{-}a+6
    \end{pmatrix}. 
  \] 
  Its positive Dirichlet group is isomorphic to $\ZZ^2$ and it is generated by
  the following two matrices:
  \[
    M_{a} \quad \text{and}\quad N_a=M_{a}^{-1}\cdot(M_{a}-I)^2.
  \]
  Consider 
  $p_{ij}=M_a^iN_a^j(0,0,1)^\top$.
  Then one of the fundamental domains of the sails is as follows:
  \begin{itemize}
    \item one vertex $p_{00}$;
    \item three edges: $p_{01}p_{00}$, $p_{01}p_{10}$, and 
      $p_{01}p_{11}$;
    \item two triangular faces:
      $p_{01}p_{00}p_{10}$ and $p_{01}p_{10}p_{11}$.
  \end{itemize}
  Here we project to the affine plane $2x+y+z=1$.
  The combinatorial situation is similar to the one of
  Example~\ref{example1} (see Figure~\ref{fig:stresses} for the case $a=0$).
  Using the formulae of Definition~\ref{def:stresscoef} we get
  \begin{equation*}
    \omega_{01,00}=\frac{a+1}{a+2},
    \qquad
    \omega_{01,10}=\frac{1}{a+2},
    \qquad
    \omega_{01,11}=\frac{a^2+3a+1}{a+2}.
  \end{equation*}
  Finally, we obtain the projection-coefficients $\beta_{ij}$ by substituting
  the three components of $p_{ij}$ to the expression $2x+y+z$.
  Then the projection-stresses for $\bar\omega$ are
  \begin{equation*}
    \begin{array}{l}
      \displaystyle
      \bar\omega_{i(j+1),ij}=\frac{a+1}{a+2}\beta_{i(j+1)}\beta_{ij};
      \\
      \displaystyle
      \bar\omega_{i(j+1),(i+1)j}=\frac{1}{a+2}\beta_{i(j+1)}\beta_{(i+1)j};\\
      \displaystyle
      \bar\omega_{i(j+1),(i+1)(j+1)}
      =\frac{a^2+3a+1}{a+2}\beta_{i(j+1)}\beta_{(i+1)(j+1)}.
    \end{array}
  \end{equation*}
\end{example}

For further examples see e.g.~\cite{karpenkov-book}.

\begin{remark}
  If $S$ is an algebraic sail, then the corresponding affine algebraic
  tensegrity is contained in the triangle of the intersection of the projection
  plane with the cone. The accumulation points of the edges of the framework
  are the boundary points of the triangle. It is not known to the authors
  whether the accumulation points of the vertices are all boundary points of
  the triangle or not.
\end{remark}

\subsection{Algorithmic aspects}
\label{subsec:algorithmic-question}

Let us now discuss some algorithmic aspects of the construction of Klein-Arnold
tensegrities and on how to reconstruct surfaces from their projections.

\subsubsection{Construction of Klein-Arnold tensegrities}

We outline the procedure that we used in the examples above to obtain planar
tensegrities by projecting three-dimensional sails to a plane.
The main part of the present paper focuses on Part~2 below.

\noindent
\emph{Input:} 
We start with an integer matrix $A \in \GL(3,\ZZ)$.

\noindent
\emph{Part 1:} \emph{Constructing periodic sails.}
\begin{itemize}
  \item First, we construct the basis of the positive Dirichlet group of
    matrices commuting with $A$. They will be used to determine period
    shifts. Finding this basis is a hard question closely related to the
    Dirichlet's unit theorem. Useful techniques can be found
    in~\cite{Buc87,DF64}.
  \item Second, we determine a period of the sail. Several techniques for that
    are described in~\cite{Oka93,Shi76} (inductive algorithms) and
    in~\cite{Kar-const} (deductive algorithm); see also~\cite{karpenkov-book}.
\end{itemize}

\noindent
\emph{Part 2:} \emph{Projecting sails to an affine plane.}

\begin{itemize}
  \item Find a suitable proper projection plane $\pi$ and project the sail
    $S(p)$ to it. As a result we obtain a framework $G_S(\bar p)$.
  \item Compute the projective lifting coefficients $\omega_S$ for one period
    (see Definition~\ref{def:stresscoef} and Proposition~\ref{integer-propos}).
  \item Determine the projection-coefficients $\beta_i$ for the vertices of the
    resulting framework and compute the projection-stress $\bar\omega_S$ by
    following Definition~\ref{def-omega-bar}. Here it is sufficient to know all
    the vertices and edges of a fundamental domain of a continued fraction and
    the basis of the positive Dirichlet group. 
\end{itemize}

\noindent
\emph{Output:} 
The Klein-Arnold tensegrity $(G_S(\bar p), \bar\omega_S)$.

\subsubsection{Techniques to reconstruct surfaces from projections and
stresses}

Let us assume that we are given a self-stressed framework in the plane
which is the projection-framework of a polyhedral surface. The goal of the
following theorem is to describe how to recover explicitly the
corresponding polyhedral surface. We call that process \emph{the
projective lifting}.

\begin{theorem}
  \label{thm:mc}
  Let $(G_S(\bar p), \bar\omega_S)$ be a self-stressed framework in the
  plane which is the projection-framework of a polyhedral surface. Let us
  assume that any pair of distinct edges of a vertex star are not parallel, i.e.,
  $p_0 p_i \not\parallel p_0 p_j$ for $i \neq j$. Then
  there exists, up to the choice of one face, a unique projective
  lifting $S(p)$ whose projection-framework is $G_S(\bar p)$ with
  self-stress $\omega_S$.
\end{theorem}
\begin{proof}
  Let us consider a vertex $\bar p_0 \in G_S(\bar p)$ and all its incident
  edges $\bar p_0 \bar p_1, \ldots, \bar p_0 \bar p_k$. Let us assume that
  they are ordered in such a way that $\bar p_0, \bar p_i, \bar p_{i + 1}$
  belong to a face. 

  Since one face of $S(p)$ is given, we can assume that $p_0, p_1, p_2$
  are given in space. Now the question is whether we can recover $p_3$
  uniquely from what we are given. 

  We know that the projection-stress coefficient $\bar\omega_{02}$
  satisfies $\bar\omega_{02} = \beta_0 \beta_2 \omega_{02}$, where
  $\bar\omega_{02}$, $\beta_0$, $\beta_2$ are given. Inserting $p_3 =
  \beta_3 \bar p_3$ into the definition of the projective lifting
  coefficient yields
  \begin{equation*}
    \bar\omega_{02} 
    = 
    \beta_0 \beta_2 \omega_{02}
    = 
    \beta_0 \beta_2 
    \frac{\det(p_1 - p_0, p_2 - p_0, p_3 - p_0)}
    {\det(p_0, p_1, p_2) \det(p_0, p_2, p_3)}
    = 
    \beta_0 \beta_2 
    \frac{\det(p_1 - p_0, p_2 - p_0, \beta_3 \bar p_3 - p_0)}
    {\det(p_0, p_1, p_2) \det(p_0, p_2, \beta_3 \bar p_3)},
  \end{equation*}
  where $\beta_3$ is the only unknown. Hence we obtain
  \begin{equation*}
    \beta_3 
    = 
    \frac{\det(p_0, p_1, p_2)}
    {\det(p_1 - p_0, p_2 - p_0, \bar p_3) - \bar\omega_{02} \det(p_0, p_1,
    p_2) \det(\bar p_0, \bar p_1, \bar p_2)},
  \end{equation*}
  and with that $p_3 = \beta_3 \bar p_3$ is uniquely determined.
  We can continue and obtain $p_4, \ldots, p_k$ one by one. This lifted
  vertex star projects to the planar vertex star with edges $\bar p_0 \bar
  p_1, \ldots, \bar p_0 \bar p_k$. Furthermore, the corresponding
  projection-stress coefficients $\tilde\omega_{01}, \ldots,
  \tilde\omega_{0,k}$ coincide, except for $\tilde\omega_{01}$ and
  $\tilde\omega_{0, k - 1}$ with the given ones as we used them above to
  reconstruct $p_3, \ldots, p_k$. Thus we have 
  $\tilde\omega_{02} = \bar\omega_{02}, \ldots, 
  \tilde\omega_{0, k-1} = \bar\omega_{0, k-1}$. 
  But what about the two remaining ones, $\bar\omega_{01}$ and
  $\bar\omega_{0k}$?

  Since both frameworks are self-stressed we have 
  $\sum\bar\omega_{0i} (p_i - p_0) = 0 
  = \sum\tilde\omega_{0i} (p_i - p_0)$.
  Furthermore, since almost everything in these two equations are equal we
  obtain
  $\bar\omega_{01} (p_1 - p_0) + \bar\omega_{0k} (p_k - p_0) = 
  \tilde\omega_{01} (p_1 - p_0) + \tilde\omega_{0k} (p_k - p_0)$, hence
  $\tilde\omega_{01} = \bar\omega_{01}$
  and
  $\tilde\omega_{0k} = \bar\omega_{0k}$ 
  since 
  $p_1 - p_0$ and $p_k - p_0$ are linearly independent by assumption.

  Now we can iteratively reconstruct all the faces. Potentially we might
  have non-zero monodromies along closed paths of consequent faces, which
  will then imply that no lifting exists but this would contradict our
  assumption. Therefore, if a lifting exists, it is reconstructible. 
\end{proof}

\begin{remark}
  The situation here is similar to Maxwell-Cremona  liftings. Such
  liftings do not always exist. This is due to monodromy conditions on the
  face-paths in the graph. 
\end{remark}

\noindent{\bf Acknowledgements.} F.\,M.\ was partially supported by the FWO grants G0F5921N (Odysseus) and G023721N, the KU Leuven iBOF/23/064 grant, and the UiT Aurora MASCOT project.
C.\,M.\ was partially supported by the Austrian Science Fund (FWF) through
grant I~4868 (SFB-Transregio “Discretization in Geometry and Dynamics”). 

\bibliographystyle{abbrv}
\bibliography{tensegrity.bib}

\bigskip 

\noindent
\footnotesize \textbf{Authors' addresses:}

\bigskip

\noindent{Department of Mathematical Sciences, University of Liverpool,
UK} 
\hfill \texttt{karpenk@liverpool.ac.uk}
\medskip  

\noindent{Departments of Mathematics and Computer Science, KU Leuven, Belgium} 
\hfill \texttt{fatemeh.mohammadi@kuleuven.be}
\medskip

\noindent{Institute of Discrete Mathematics and Geometry, TU Wien, Austria} 
\hfill \texttt{cmueller@geometrie.tuwien.ac.at}
\medskip  

\noindent{School of Mathematical Sciences, Lancaster University, Lancaster, UK} 
\hfill \texttt{b.schulze@lancaster.ac.uk}

\end{document}